\documentclass[11pt]{article}
\usepackage[mathscr]{euscript}

\usepackage{amsmath,amssymb}
\usepackage{graphicx}
\usepackage{bbm}
\usepackage{xcolor}
\usepackage{stmaryrd}
\usepackage{upgreek}
\usepackage{esint}
 \newtheorem{remark}{Remark}
\newtheorem{theorem}{Theorem} \newtheorem{lemma}{Lemma}

\newtheorem{proposition}{Proposition}

\def \R{{\mathbb{R}}}

\def \C{{\mathbb{C}}}
\def \Z{{\mathbb{Z}}}

\def \vae {{\varepsilon }}
\def \fe {{f^\varepsilon}}

\def \torus  {{\mathbb{T}}^d }
\def \del {\partial} 

\title{About the derivation of the quasilinear approximation in plasma
  physics.}  \author{Claude Bardos  \footnotemark[1] \and Nicolas
  Besse \footnotemark[2]} 

\begin{document}

\maketitle \tableofcontents


\renewcommand{\thefootnote}{\fnsymbol{footnote}}

\footnotetext[1]{ Laboratoire J.-L. Lions, Universit\'e Pierre et
  Marie Curie - Paris 6   BP 187, 4 place Jussieu, 75252 Paris, Cedex
  5, France ({\tt claude.bardos@gmail.com})}

\footnotetext[2]{ Laboratoire J.-L. Lagrange,  UMR CNRS/OCA/UCA 7293,
  Universit{\'e} C\^ote d'Azur, Observatoire de la C\^ote d'Azur, Bd
  de l'observatoire CS 34229, 06300 Nice, Cedex 4, France.  ({\tt
    Nicolas.Besse@oca.eu})}

\maketitle
\begin{abstract}
 This contribution, built on the companion paper \cite{BB}, is
  focused on the different mathematical approaches available for the
  analysis of the quasilinear approximation in plasma physics.
\end{abstract}

\section{Introduction and notation}
The origin of this contribution is the issue of the approximation of
solutions of the Vlasov equation 
\begin{equation*}
\del_t f +v\cdot \nabla_x f + E\cdot \nabla_v \cdot f =0
\,,\label{Liouville00}
\end{equation*}
where $f(t,x,v)$ is a probability density driven by a selfconsistent
potential, given in terms of this density by the Poisson equation
\begin{equation*}
-\Delta \Phi(t,x)  = \rho(t,x) =\int_{\R^d_v} f(t,x,v)dv -1 \,, \quad
E(t,x) = -\nabla \Phi(t,x)\,,
\end{equation*}
in the  domain $\torus \times \R^d_v$,   with $\torus=({\R_x}/{2\pi\mathbb Z})^d$,
by a parabolic
(linear or nonlinear) diffusion equation for the space averaged
density of particles, namely
\begin{equation}
\del_t\overline{f}(t,v)-\nabla_v\cdot \big(\mathbb D(t,v) \nabla_v
\overline{f}(t,v)\big) =0 \,.\label{diffusion0}
\end{equation}
Such equation carries the name of quasilinear approximation  and is a
very active subject of plasma physics. Here,
relying on a compagnon paper \cite{BB} (devoted to a more detailed
physical analysis and more focused on the interpretation of
turbulence), we focus on the different mathematical approaches
motivated by all the recent progress (for instance around the question
of Landau damping) on the analysis of the Vlasov equation.

Starting from the natural scaling derived for instance in \cite{BB},
we propose a rescaled version of the Vlasov equation and first show
obstructions to the convergence to  an equation of the type
(\ref{diffusion0}) with a non zero diffusion. This leads to the
introduction of a stochastic vector field, hence to a non
selfconsistent Liouville equation. There, a direct approach (as in the
contributions of A. Vasseur and coworkers \cite{LV04, PV03})  produces
a complete positive answer.

A more classical analysis leads to the comparison with the present
results on Landau damping and to an alternate approach based on the
spectral theory and at variance with the rescaled equation valid only
for short time. No complete proof is given but a natural road for
convergence based on some plasma physics computations is proposed.

As a short time correction, it may play in the subject  the same role
as was done by the introduction of the diffusion in the macroscopic
limit of the Boltzmann equation by Ellis and Pinsky \cite{EP75}.

As a mathematical contribution to physics (however modest it may be)
we dedicate this paper to the memory of Alex Grossman. Besides being
recognized for super scientific achievement in particular with the
introduction of wavelets, he will be remembered over many years
with his generous and charismatic influence on our community.

\subsection{Notation and some hypotheses}
The flow $S(t): f\mapsto  f(t,x-vt,v)\,$,  with $s\mapsto x-vs$
denoting the free advection flow modulo $(2\pi)^d$ on $\torus$, is the
advection flow generated by the operator $-v\cdot \nabla_x\, $.  In
the same way, we also introduce the flow $S^\vae$ generated by the
operator $-\vae^{-2}v\cdot \nabla_x\, $ and defined by
\begin{equation*}
S^\vae_t f=S\left(\frac{t}{\vae^2}\right) f
=f\Big(t,x-v\frac{t}{\vae^2},v\Big)\,. \label{advection0}
\end{equation*}
These are unitary groups in any $L^p(\torus\times \R^+_v)$, with $1\le
p\le\infty$, which preserve the positivity and the total mass.
%%%Pour simpliifer j'ai laissé tomber toute allusion a un domaine $X$ different de $\torus$
Since this is not the relevant issue for our discussion, in presence
of an $\vae>0\,, $  the initial data $f^\vae(0,x,v)$ is assumed to be
independent of $\vae$ and as smooth as required (hence taking in
account existing regularity results \cite{GLA96} for the Vlasov
equation) to have global in time solutions which will satisfy the
relevant computations. On the other hand, emphasis has to be put on
the regularity estimates which are independent of $\vae\,.$

As observed in many previous publications starting from Landford
\cite{LAND75}, limit can be obtained not at the level of the equation
but at the level of the solution itself. As a consequence, we will use
the first order Duhamel formula
\begin{equation*}
 f(t,x,v) =S_tf(0) -\int_0^t d\sigma_1\, S_{t-\sigma_1}E(\sigma_1)\cdot\nabla_v
 f(\sigma_1)\,,
\end{equation*}
and in Section~\ref{Sec:NSCSA} an avatar of the second order Duhamel,
to connect the value of $f(\sigma_1)$ with the value of $f(\sigma_2)$
according to the Duhamel formula
\begin{equation*}
 f(\sigma_1)=S_{\sigma_1-\sigma_2}f(\sigma_2)
 +\int_{\sigma_2}^{\sigma_1} d\sigma\, S_{\sigma_1-\sigma} E(\sigma)
 \cdot\nabla_v f(\sigma)\,,
\end{equation*}
which finally gives
\begin{equation*}
\begin{aligned}
&f(t) =S_tf(0) +\int_0^t d\sigma_1\, S_{t-\sigma_1} E(\sigma_1)\cdot
  \nabla_v S_{\sigma_1-\sigma_2} f(\sigma_2)\\ &+ \int_0^t
  d\sigma_1\int_{\sigma_2}^{\sigma_1} d\sigma\, S_{t-\sigma_1}
  E(\sigma_1)\cdot \nabla_v(S_{\sigma_1-\sigma} E(\sigma)\cdot\nabla_v
  f(\sigma))\,.
\end{aligned}
\end{equation*} 
Denoting by $\fint   dx$ the $x$-average on $\torus$,  one
obtains 
\begin{equation}
\del_t \fint f(t,x,v)dx + \nabla_v \cdot \fint E(t,x)
f(t,x,v) dx =0 \label{Fick0}\,.
\end{equation}
The second term of  (\ref{Fick0}) is the divergence of the averaged
flux 
\begin{equation*}
J = \fint E(t,x) f(t,x,v) dx \,.
\end{equation*}
Therefore almost all the rest of this contribution is devoted to the
determination of such flux (sometimes called a ``Fick law").  Since
weak convergence involves distributions  and test functions, such
duality is denoted by the following bracket notation, $\langle
.,.\rangle$. Using the fact that $\mathcal D (\R_t\times
\torus\times\R^d_v  )= \overline{ \mathcal D ( \R_t ) \otimes \mathcal
  D (\torus )\otimes \mathcal D(\R^d_v)} $, test functions depending
on  one or an other of these spaces will be used according to
convenience. Eventually the symbol $\overline{T^\vae}$ will be used to
denote any cluster point  (in the sense of distributions or under some
other stronger topology) of a family $\{T^\vae\}$ of bounded
distributions.

\subsection{The rescaled Liouville equation}

Both from plasma physics considerations (see \cite{BB}) and also
because it is compatible with the scaling invariance of the diffusion
in the velocity variable (see equation (\ref{diffusion0}), densities $f^\vae(t,x,v)$,
solutions of the following rescaled (with $\vae>0$) Liouville
equation, 
 \begin{equation} \label{eqn:L1}
 \partial_tf^\varepsilon +
 \frac{v}{\varepsilon^2}\cdot\nabla_xf^\varepsilon +
 \frac{E^\varepsilon}{\varepsilon}\cdot \nabla_{v}f^\epsilon =0\,,
 \quad \hbox{with}\quad  E^\vae =-\nabla \Phi^\vae\,,
\end{equation}
are considered.

As it will soon appear below, the specific behavior of the solutions
as $\vae \rightarrow 0$ (and/or as $t\rightarrow \infty$) depends of
the time singularities of the potential rather than the space
regularity and properties of the initial data.  Then, unless
otherwise specified, it is assumed below that the initial data
$f_0\in\mathcal S(\torus\times\R^d_v)$  is an $\vae$-independent
smooth function  and that   $E^\vae=-\nabla \Phi^\vae\,$ is, locally
in time, uniformly Lipschitz with respect to the variable $x$. 
% Je crois qu'a ce niveau  pour eviter de tomber dans le Lions DiPerna cet hypothèse est raisonnable.
Under such hypotheses, solutions of  equation (\ref{eqn:L1}) are well
and uniquely defined.  On the other hand no assumption is made on the
uniform regularity of the solution $\fe(t,x,v)$ either as
$t\rightarrow \infty$ or as $\vae\rightarrow 0$. 

Hence, only $\vae -$independent estimates, which are in agreement with
the classical results including in particular those concerning the
solution of the Vlasov equation, 
% Dans Vlasov il y au moins 2 equations Donc Vlasov equation ou Vlasov equations ???
are based on the fact that the Liouville equation preserves positivity
and Lebesgue measure.
\begin{equation*}
\begin{aligned}
& \forall \, t\in\R^+, \quad 0\le \fe(t,x,v)\le \sup_{(x,v)\in
    \torus\times\R^d_v} f_0(x,v), \quad \int_{\torus\times\R^d_v}
  \fe(t,x,v) dxdv =1\,,\\ &\forall 1 \le p\le \infty,\quad \forall \,
  t\in\R^+, \quad \| f^\vae(t)\|_{L^p(\torus\times \R^d_v)}= \|f(0)
  \|_{L^p(\torus \times \R^d_v)} \label{apriori0}\,.
\end{aligned}
\end{equation*}
To use the scaling and the ergodicity of the $d-$dimensional torus the
following proposition is recalled.
\begin{proposition}{(Ergodicity)}
  \label{lem:EFF}
 \begin{enumerate}
\item  Any  $g\in L^p(\torus\times \R^d_v)$, with $1\leq p \leq
  \infty$, which satisfies the relation 
  \begin{equation}
v\cdot \nabla g =0, \quad in \quad  \mathcal{D}'({{\mathbb{T}}^d
}\times \R^d_v)\,, \label{erg1}
\end{equation}
is an $x$-independent function.
\item The solutions $\fe$ of the equation
\begin{equation}
\vae^2 \del_t \fe + v\cdot\nabla_x \fe =0\,, \quad
\fe(0,x,v)=f_0(x,v)\in L^p(\torus\times\R^d_v)\,, \label{CM2}
\end{equation}
converge in   $L^\infty(\R^+_t ; \mathcal D'(\R^d_v))$ to $\fint
f_0(x,v)dx\,.$
\end{enumerate}
\end{proposition}
This proposition and its use in Landau damping are classical (see
\cite{CM98, HV09}) and, for sake of completeness,  its proof is
shortly recalled below.
\begin{proof}
From  (\ref{erg1})  one deduces the relation,
$$ \frac d{dt} g(x-vt,v)=0\,,
$$ which by Fourier transform gives 
  \begin{equation*}
    \label{eqn:lemEFF:1}
    \big(1-\exp({\rm i}   k\cdot v t   )\big) \hat{g}(k,v)=0, \quad
    \forall k\in \Z^d, \ v\in \R^d, \ t \in \R\,.
  \end{equation*}
  This relation implies that the support of $\hat g$ is contained in
  the set
  \[
    {\rm supp}(\hat g):=\{(k,v)\in\Z^d \times \R^d \ |\ k\cdot v t
    \in 2\pi\Z, \ \forall t\in\R\}\,.
  \]
  For any $\delta, \, T, \, r, \, R > 0$, such that $ \delta <T$,
  and $r<R$, the Lebesgue measure of the set ${\rm supp}(\hat g)$ for $\delta < t < T$
  and $ r <|v| <R$ is zero. Therefore, Since $ g$ belongs to
  $L^p(\torus\times\R^d_v)$, this forces $\hat g$ to be equal to zero
  for all $k\neq 0$  and finally one obtains
  $$g=\hat g(0,v) =\fint  g(x,v) dx \,.$$
 In the same way for  the point 2, use the fact that the solution of
(\ref{CM2}) is given by $\fe(t,x,v) =f_0(x-v t/{\vae^2} ,v)$ to
 write, for all $ k\in\mathbb Z^d $ and for any $ \phi \in \mathcal
 D(\R^d_v)\,,$
  \begin{equation} 
 \int_{\R^d_v}\hat{\fe} (t,k,v)   \phi(v)dv=\int_{\R^d_v}dv\fint
 dx \fe(t,x,v)e^{-ik\cdot x } \phi(v)dv=\int_{\R^d_v} \hat {f_0}(k,v)
 \phi(v) e^{ik\cdot v \frac{t}{\vae^2}}dv\,. \label{CM1}
\end{equation}
  Since $f_0(x,v)\in L^p(\torus \times \R^d_v)\,,$ by the Riemann
  Lebesgue theorem the right-hand side of (\ref{CM1}) goes to $0$ for
  $k\not=0$ as  $\vae\rightarrow 0\,.$ Hence completing the proof of
  the point 2.
\end{proof}
  
As a consequence, writing the rescaled Liouville equation in the
following form, 
 \begin{equation} \label{eqn:RL1}
v \cdot\nabla_xf^\varepsilon= -\vae^2 \partial_tf^\varepsilon -\vae
E^\varepsilon \cdot \nabla_{v}\fe,  \end{equation} one deduces from
 the uniform estimates  that any cluster point $\overline
 {\fe}=\overline{\fe}(t,v)$ of the family $\{\fe\}\,,$ in the $
 L^\infty(\R^+_t\times\torus\times\R^d_v)$ weak$-\star$ topology, is
 independent of $x$ and is a solution of the equation 
\begin{equation}\label{Fick 0}
\del_t \overline{\fe} + \nabla_v\cdot  \left( \overline{ \fint {
    \frac{E^\vae \fe} {\vae}}dx}\right)=0\,.
\end{equation}
In (\ref{Fick 0}) the Fick law (relating the variation of the density
to the divergence of the current) appears as the ratio of two terms
going  {\bf at least formally} to zero, because under the hypothesis
$\overline{ E^\vae \fe}=\overline{ E^\vae}\,\overline{ \fe}$ one has
\begin{equation*}
\overline{\fint E^\vae \fe dx}= -\fint \nabla
\overline{\Phi^\vae}(t,x)\overline{\fe}(t,v) dx=0\,.
\end{equation*}
The justification of the quasilinear approximation would be the proof
that  
\begin{equation}
 \overline{ \fint \frac {  E^\vae \fe} {\vae} dx} = -\mathbb
 D(t,v) \nabla_v \overline{\fe}\,,
 \label{FickDiff}
\end{equation} 
with $\mathbb D (t,v)$ being  a non negative diffusion matrix.

\subsection{Obstruction to the convergence  to a non degenerate diffusion matrix}
The Liouville equation is the paradigm of an Hamiltonian system while
the diffusion equation is the model of an irreversible phenomena. Then
a paradox comes from  the comparison of the two equations 
\begin{equation}\label{paradox}
  \frac 12 \frac{d}{dt}\int_{\R^d_v} dv \fint dx \,
  |\fe|^2 =0 \quad \hbox{and}\quad  \frac 12  \frac{d}{dt}
  \int_{\R^d_v}dv \fint dx \, |\overline{\fe}|^2
  +\int_{\R^d_v}dv \fint dx \, \big(\mathbb D(t,v) \nabla_v
  \overline{\fe} \,,\, \nabla_v \overline{\fe}\big) =0\,.
 \end{equation}
This paradox has to be resolved to justify, in some cases, that the diffusion matrix
$\mathbb D$ is not degenerate, taking in account for instance the
following 
\begin{proposition}\label{preremarques}
Assuming that  $E^\vae=-\nabla \Phi^\vae$ is  uniformly bounded in
$L^\infty(0,T ; W^{1,\infty}(\torus))\,,$ and that $ {\fe}(t,v)$ is
solution of (\ref{eqn:RL1} )with  $\overline{\fe(0,x,v)}=\fint
f_0^\vae(x,v)dx$, one has the following facts.
\begin{enumerate}
\item The density $f^\vae(t,x,v)$   converges strongly  in
  $L^2([0,T]\times\torus \times\R^d_v )$ if and only if 
\begin{equation*}
\int_0^Tdt \int_{\R_v^d} dv \, \big(\mathbb D(t,v)\nabla_v
\overline{\fe}(t,v) \,,\, \nabla_v \overline{\fe}(t,v)  \big)=0\,.
\end{equation*}
\item If $\del_t \Phi^\vae (t,x)$ is bounded in some distribution
  space $L^1(0,T; H^{-\beta}(\torus))$ with some $\beta$ finite and
  if $\vae\del_t \Phi^\vae$ converges to $0$ in $L^1([0,T]\times
  \torus)$ then one obtains
\begin{equation*}
\del_t\overline{\fe}=0 \quad \hbox{and}\quad  v\cdot \overline{
  \fint \frac {  E^\vae \fe} {\vae} dx} =0 \quad \hbox{on} \quad
    [0,T]\times \R^d_v\,.
\end{equation*}
\end{enumerate}
\end{proposition}
\begin{proof}
The point 1 is a direct consequence of the comparison between the
classical Hilbertian estimate 
\begin{equation*}
  \forall t \in (0,T), \quad \int_0^t  \overline{ \|\fe(s)
    \|^2_{L^2(\torus\times \R^d_v)}} ds\ge \int_0^t \|\,\,
  \overline{\fe(s)}\,\, \|^2_{L^2(\torus\times \R^d_v)} ds\,,
\end{equation*}
and the two equations appearing in the formula (\ref{paradox}).
%Pas envie de detailler plus et pas necessaire d'invoquer la semi continuite
For the point 2,  multiply the rescaled Liouville equation
by $\vae \Phi^\vae(t,x)\theta(t) \phi(v) $ to obtain after
integration 
\begin{equation}
\begin{aligned}
&\int_{\R^+_t}\int_{\R^d_v}  \theta (t) \phi(v)\fint
  \frac{v\cdot\nabla_x \fe(t,x,v) \Phi^\vae(t,x) dx}{\vae}dvdt  \\ & =
  \int_{\R^+_t} \int_{\R^d_v} \theta (t)\nabla_v  \phi(v)\cdot
  \fint   E^\vae(t,x)  \Phi^\vae(t,x) \fe(t,x,v) dx dvdt
  \\ &+\vae \int_{\R^+_t} dt\fint dx \int_{\R^d_v} dv  \,\fe
  (t,x,v) \del_t (\theta (t) \Phi^\vae(t,x))\phi(v)
  \,.\label{trivial1}
\end{aligned}
\end{equation}
Now, since $\vae \del_t\Phi^\vae(t,x)$ converges strongly to zero in  $L^1([0,T] \times \torus )\,$
as $\vae \rightarrow 0\,,$ the last term of the
right-hand side of  (\ref{trivial1}) goes to $0$ as $\vae\rightarrow 0 $,
while for the first term of the right-hand side of  (\ref{trivial1}) with the Aubin-Lions theorem (see for
instance \cite{Sim87})  one obtains
\begin{equation}
\begin{aligned}
&\overline{ \int_{\R^+_t}\int_{\R^d_v} \theta (t)\nabla_v
    \phi(v)\cdot \fint   E^\vae(t,x)  \Phi^\vae(t,x) \fe(t,x,v)
    dx dvdt}\\ &= - \int_{\R^+_t} \int_{\R^d_v}\theta (t)\nabla_v
  \phi(v)\cdot \fint   \frac 12 \overline{ \nabla_x
    |\Phi^\vae(t,x)|^2} \overline{\fe(t,x,v)} dx dvdt
  \\ &=-\int_{\R^+_t}\int_{\R^d_v} \theta (t)\overline{\fe}(t,v)
  \nabla_v\phi(v)\cdot \fint   \frac 12 \nabla_x \overline{
    |\Phi^\vae(t,x)|^2}  dx =0\,.  &
\end{aligned} \label{trivial2}
\end{equation} 
Eventually, using  $E^\vae=-\nabla \Phi^\vae$ and an integration by
parts in the variable $x$, for the left-hand side of (\ref{FickDiff}) we obtain from
(\ref{trivial1})-(\ref{trivial2}),
\begin{equation}\label{trivial4}
  \begin{aligned}
    &\overline{ \int_{\R^+_t}\int_{\R^d_v} \theta (t) \phi(v)  v\cdot
      \fint \frac{  E^\vae(t,x)   \fe(t,x,v)}{\vae} dx
      dvdt}\\ &=-\overline{ \int_{\R^+_t}\int_{\R^d_v} \theta (t)
      \phi(v) \fint   \frac{v\cdot   \nabla_x \Phi^\vae(t,x)
        \fe(t,x,v)}{\vae}  dx dvdt}\\ &=
    \overline{\int_{\R^+_t}\int_{\R^d_v}   \theta (t) \phi(v)
      \fint \frac{v\cdot \nabla_x \fe(t,x,v)
        \Phi^\vae(t,x)}{\vae}  dx dv dt} =0\,.
  \end{aligned}
\end{equation}
Then, one observes that any vector-valued function $v\mapsto \psi(v)
\in \mathcal D(\R^d_v; \R_v^d)$, with $\psi(0)=0\,$, can be written
with the introduction of a function $v\mapsto \gamma (v)\in \mathcal
D(\R^d_v)$ equal to $1$ on the support of $\psi(v)$ as
\begin{equation}
\begin{aligned}
&\psi(v)= \left(\int_0^1 \nabla_v\psi(sv) ds\right) v =  \gamma
(v)\left(\int_0^1 \nabla_v\psi(sv) ds\right)v= \varphi(v) v\\
&\hbox{
  with }\,  \, \varphi(v):= \gamma(v)\int_0^1\nabla_v\psi(sv)
ds\,. \label{trivial3}
\end{aligned}
\end{equation}
From (\ref{trivial4})-(\ref{trivial3}) one concludes that for any such
vector-valued function $v\mapsto \psi(v)$ with $\psi(0)=0\,,$ one obtains
\begin{equation*}
 \int_{\R^+_t} \theta (t)  \int_{\R^d_v}  \psi(v) \cdot \overline{
   \fint \frac{  E^\vae(t,x)   \fe(t,x,v)}{\vae} dx} dvdt =0\,.
\end{equation*}
Therefore the support of $\del_t  \overline{\fe}(t,v)$ is contained in
$[0,T] \times \{v=0\}\,.$  Hence for any $\theta (t) \in \mathcal D
(\R_t^+)$ and for any $\phi(v)\in \mathcal D(\R^d_v)$ with the point
$0$ not included in the support of $\phi(v)$, one obtains 
\begin{equation} 
\int_{\R^+_t\times \R^d_v}\overline{\fint f^\vae dx} \del_t
\theta (t)  \phi(v)dvdt=0\,. \label{stupide}
\end{equation}
However, since $\overline{\fint f^\vae dx} \in
L^\infty(\R^+_t\times \R^d_v) $ relation (\ref{stupide}) remains valid
for any test function $\phi\in \mathcal D(\R^d_v)$ and this completes
the proof of the point 2.
\end{proof}
\begin{remark}
For Vlasov--Poisson equations the ``ergodic" convergence of $\, \fe
(t,x,v)\, $ to $ \,\overline{\fe(t,v)}\, $, which has already been proven,
implies (using the Poisson equation) the weak convergence to zero of the electric field
$E^\vae(t,x)\,$. This is in this scaling the ``baby Landau damping".
Strong convergence will be equivalent to the genuine Landau damping as
proven in Theorem~\ref{babydamping}  of Section~\ref{bbd}.
 As a consequence for the electric  field $E^\vae\,,$ time regularity
 may  prevent a limit described by a non degenerate diffusion
 equation.
\end{remark}
\begin{remark}
From the above observations one concludes that, as it is the case in
many related examples,  proofs should rather involve the behavior of
the solution itself rather than the asymptotic structure of the
equation. Therefore the use of the Duhamel expansion  (up to
convenient order) appears to be a natural tool and it will appear in
two very different approaches. The first one is based on the
introduction of stochasticity  in the electric field (see
Section~\ref{Sec:NSCSA}). Hence it corresponds to a situation where
such electric field is non selfconsistent: it is a genuine Liouville
equation.  The second approach, based on a short time asymptotic (see
Section~\ref{RMQA}), deals with configurations where the electric
field is given selfconsistently (i.e. determined from the density of
particles) through a spectral analysis.
\end{remark}
 
\subsection{The first iteration of the Duhamel formula and the
  diffusion - Reynolds electric stress tensor}
From the previous section one deduces that the convergence to a
genuine diffusion equation requires that the vector field $E^\vae$ or
the potential $\Phi^\vae$ becomes ``turbulent" as $\vae\rightarrow0\,$.
This justifies at present the introduction of the diffusion
tensor $\mathbb D^\vae$. Using the Duhamel formula
\begin{equation}
  \label{eqn:PDF}
  f^\varepsilon(t)=  S_t^\varepsilon f_0^\varepsilon - \frac1\vae
  \int_0^t  S_{t-\sigma}^\varepsilon E^\vae (\sigma)\cdot \nabla_v
  f^\varepsilon(\sigma) d\sigma\,,
\end{equation}
and the ergodicity (see Proposition~\ref{preremarques}), the first
term of the right-hand side of (\ref{eqn:PDF}) is ignored, while
multiplying by a test function $\phi(v)\,,$ one obtains for the Fick term
the following expression,
\begin{equation}\label{weakfubini}
-\overline{\fint dx \int_{\R^d_v}dv \,\phi(v)  \nabla_v\cdot
  \left( \frac{E^\varepsilon
    f^\varepsilon}{\varepsilon}\right)}=-\frac{1}{\vae^2}
\overline{\int_{\R^d_v} dv\, \nabla_v \phi(v) \cdot \int_0^tds\,
  \fint dx\, E^\vae(t)   S_{t-s}^\varepsilon (E^\vae
  (s)\cdot\nabla_v f^\varepsilon(s))}\,.
\end{equation}

Using the explicit formula,
\begin{equation}
S^\vae_{t-s}(E^\vae(s)\cdot \nabla_v f^\vae(s))= E^\vae(s,
x-v(t-s)/\vae^2)\cdot (\nabla_v f^\vae)(s,x-v(t-s)/{\vae^2},v)\,,
\end{equation}
the change of variable $\sigma= ({t-s})/{\vae^2}$, and  the
$2\pi-$periodicity of the functions $E^\vae$ and $\fe$, one obtains from (\ref{weakfubini})
\begin{equation}\label{biduhamel}
\begin{aligned}
&-\overline{\fint dx \int_{\R^d_v}dv\, \phi(v)   \nabla_v \cdot
    \left(\frac{E^\varepsilon f^\varepsilon}{\varepsilon}\right)}\\ &=
 \! \overline{\int_{\R^d_v} dv\, (\nabla_v \phi(v))^T
    \int_0^{\frac{t}{\vae^2}} d\sigma \fint dx \, E^\vae(t,x+
    \sigma v)\otimes E^\vae (t-\vae^2\sigma,x) \nabla_v f^\varepsilon(
    t-\sigma\vae^2,x,v)}\\ 
    &=  \! \overline{\int_{\R^d_v}dv
    \fint dx  \int_0^{\frac{ t}{\vae^2}}d\sigma
    f^\varepsilon(t-\sigma\vae^2, x,v   ) \nabla_v \cdot \big(  E^\vae
    (t-\vae^2\sigma,x)   \otimes  E^\vae(t,x+ \sigma v)    \nabla_v
    \phi(v) \big)}\,.
 \end{aligned}
\end{equation}
Observe that the above integrations by part are justified on the
following ground: \\
i) All the arguments, $E^\vae$ and $\fe$, are
assumed for $\vae>0$ to be smooth enough.  \\
ii) The smooth test function $\phi$ is independent of $x$ and $t$.\\
%\\
%iii) Concerning the function $\fe$ one has
%\begin{equation}
%\nabla_v(\fe (t-\sigma \vae^2,x ,v))= (\nabla_v\fe) (t-\sigma \vae^2,x ,v)
%\end{equation}

Further analysis may be decomposed in four steps.
 
\begin{enumerate}
\item In (\ref{biduhamel}) replacing $\fe$ by an other smooth test
  function, one introduces the diffusion tensor $\mathbb D^\vae(t,v)$ defined by
\begin{equation*}
\begin{aligned}
&\int_{\R^d_v} dv\,  \big(\mathbb D^\vae (t,v)\nabla_v \psi(v) \,,\, \nabla_v
  \phi(v) \big) =\int_{\R^d_v} dv\,  \big(\nabla_v \psi(v) \,,\, \mathbb
  D^\vae(t,v)^T\nabla_v \phi(v) \big) \\ &=\int_{\R^d_v}dv \,    (\nabla_v
  \psi(v) )^T\Big(\fint dx  \int_0^{\frac{ t}{\vae^2}}
  d\sigma\ E^\vae (t-\vae^2\sigma,x) \otimes E^\vae(t,x+ \sigma v)
  \Big)\nabla_v  \phi(v)\,.
\end{aligned}
\end{equation*}
\item Eventually observe that under a decorrelation hypothesis and
  with 
\begin{equation*}
\overline{ f^\varepsilon(t-\sigma\vae^2, x,v)} =\overline{\fe}(t,v)\,,
\end{equation*}
one would obtain
\begin{equation*}
  \label{stresstensor}
 \begin{aligned}
&\overline{\int_{\R^d_v}dv  \fint dx
     \int_0^{\frac{ t}{\vae^2}}d\sigma  f^\varepsilon(t-\sigma\vae^2
     ,x,v) \nabla_v \cdot \big( E^\vae (t-\vae^2\sigma,x) \otimes
     E^\vae(t,x+ \sigma v) \nabla_v  \phi(v) \big)} \\ &=
   \int_{\R^d_v}    \overline{ f^\varepsilon }(t,v) \nabla_v \cdot
    \big(\overline{ \mathbb D^\vae(t,v)^T}  \nabla_v  \phi(v) \big)
   dv \,.
\end{aligned}
\end{equation*}
  
\item With the notation 
$$ E^\vae (t-\vae^2\sigma,x) \otimes E^\vae(t,x+ \sigma v)=E(T,X)
  \otimes E(S,Y)_{T=t-\vae^2\sigma,\, X=x,\, S=t,\, Y=x+\sigma v}\,,
$$ appears some type of average of the tensor which in the present
  field ("plasma turbulence") plays a role very similar to the
  Reynolds hydrodynamic stress tensor $u(t,x)\otimes u(s,y)$ in fluid turbulence. 

\item Since  the electric field is the gradient of a real $x$-periodic
  potential, i.e.  $E^\vae(t,x)=-\nabla\Phi^\vae(t,x)$,  with the
  notation
\begin{equation*}
E^\vae(t,k)=\fint E^\vae (t,x)e^{-ik\cdot x} dx\,\quad
\hbox{and} \quad\Phi^\vae(t,k)=\fint \Phi(t,x)e^{-ik\cdot x} dx\,,
\end{equation*}
one has
\begin{equation*}
 \mathbb D^\vae(t,v) =\sum_{k\in \mathbb Z^d\backslash\{0\}}k\otimes
 k\int_0^{\frac{ t}{\vae^2}}   \Phi^\vae(t-\vae^2\sigma,k)
 (\Phi^\vae(t ,k) )^\star e^{-ik\cdot v\sigma}d\sigma\,,
\end{equation*}
with 
\begin{equation*}
  \sum_{k\in \mathbb Z^d} |k|^4|\Phi^\vae(t,k)|^2 \le C\ \hbox{
    independent of }\vae\,.
\end{equation*}
\end{enumerate}
In order to compare deterministic results of this section with the
stochastic ones of Section~\ref{Sec:NSCSA}, we state and prove the
following
\begin{proposition}
  \label{prop:DetLim}
We assume  that the Fourier coefficients of the potential
$\Phi^\vae(t,x)$ are given by the ansatz
\begin{equation*}
\Phi^\vae(t,k) ={\underline\Phi}^\vae(t,k)e^{-i \omega(k)
  \vae^{-\beta_k} t}\,,
\end{equation*}
where $ \omega(k) \vae^{-\beta_k}$ is a fast time frequency, while the
amplitude $\underline{\Phi}^\vae(t,k)$ is slowly varying with time and
more precisely satisfies the estimate
\begin{equation}
  \label{RegDec}
\sum_{k\in \mathbb Z^d} |k|^4|{\underline\Phi}^\vae(t,k)|^2 \le
C\  \hbox{ independent of } \vae\,.
\end{equation}
Then,
\begin{enumerate}
\item The diffusion tensor $\mathbb D^\vae$ is  given by
\begin{equation*}
\begin{aligned}
\mathbb D^\vae(t,v)&  =\fint dx
\int_0^{\frac{ t}{\vae^2}}d\sigma \, E^\vae(t,x+ \sigma  v)
\otimes    E^\vae (t-\vae^2\sigma ,x)\\ &=\sum_{k\in\mathbb Z^d}
k\otimes k \ \int_0^{\frac{ t}{\vae^2}}d\sigma\,
|\underline{\Phi}^\vae(t,k)|^2  e^{ -i(\vae^{2-\beta_k}\omega(k) -
  k\cdot v)\sigma} \\ &=\sum_{k\in\mathbb Z^d} k\otimes k
\  |\underline{\Phi}^\vae(t,k)|^2 \  \frac{
  \sin\big((\vae^{2-\beta_k}\omega(k) - k\cdot v)\frac t{\vae^2}\big)}
   {\vae^{2-\beta_k}\omega(k) - k\cdot v}\,.
\end{aligned}     
\end{equation*}
\item Using the definition of the following hyperplanes,
  \[
  \pi_k^0 = \{ v \in \R^d_v \ \hbox{ such that } \ k\cdot v=0 \}\,,
  \]
  and
  \[
  \pi_k^1 = \left\{ v \in \R^d_v \ \hbox{ such that }
  \ \omega(k)-k\cdot v=0 \ \mbox{ or } \  k\cdot(\vec{\omega}(k)-v)=0
  \ \mbox{ with } \ \vec{\omega}(k): = \frac{\omega(k)}{|k|}
  \frac{k}{|k|} \right\}\,,
  \]
  for any $\psi, \, \phi \in \mathcal D(\R^d_v) $ one obtains the
  following behavior as  $\vae\rightarrow 0\,$:
\begin{equation*}
\begin{aligned} 
&\int_{\R_v^d}(\nabla_v \phi(v))^T\,\overline{ \mathbb
    D^{\vae}(t,v)}\nabla_v\psi(v) \,dv =  \pi \sum_{k\not=0,\,
    \beta_k<2}  \overline{|\underline{\Phi}_k^\vae |^2} \int_{\pi_k^0}
  k\cdot \nabla_v\phi(v)\, k\cdot \nabla_v\psi(v)\, dv \\ &  +\pi
  \sum_{k\not=0,\, \beta_k=2} \overline{  |\underline{\Phi}_k^\vae|^2
  } \int_{\pi_k^1} k\cdot\nabla_v\phi(v)\, k\cdot
  \nabla_v\psi(v)\,dv\,.
\end{aligned}
\end{equation*}
\end{enumerate}
\end{proposition}
\begin{proof}
  As above the proof uses the Fourier expansion of the vector field
  $E^\vae(t,x)=-\nabla \Phi^\vae (t,x)\,$, the limit in the sense of
  distribution of the function $\sin({s}/{\vae})$, and finally the
  fact that for $\beta_k>2$ one obtains
  \[
  \lim_{\vae\rightarrow 0} \frac{
    \sin\big((\vae^{2-\beta_k}\omega(k)-k\cdot v)\frac t{\vae^2}\big)}
      {\vae^{2-\beta_k} \omega(k)-k\cdot v  }=0\,.
  \]
\end{proof}

\begin{remark}
  In point 2 of Proposition~\ref{prop:DetLim}, one observes
  that under the decreasing condition (\ref{RegDec}), the cluster point
  $\overline{\mathbb D^\vae}$ is the distribution
  \[
  \overline{\mathbb D^\vae}=
  \pi \!\!\!\! \sum_{k\in \Z^d, \, \beta_k<2}  k\otimes k
  \  |\underline{\Phi}^\vae(t,k)|^2\, \delta(k\cdot v)
 \ + \
   \pi \!\!\!\!  \sum_{k\in \Z^d, \, \beta_k=2}  k\otimes k
  \  |\underline{\Phi}^\vae(t,k)|^2\, \delta(\omega(k) - k\cdot v).
  \]
  In the next section, the introduction of stochasticity in the
vector field (or in the potential)  has in particular the effect of
smoothing the diffusion kernel by the introduction of regularized
densities of 
$$ \delta(\omega(k)-k\cdot v)= \delta(k\cdot(\vec{\omega}-v)),
$$ as it appears in Section~4.4 of \cite{BB} for the ``resonance broadening''--like
approximation of the
``phase velocity"  $\omega(k)/|k|\,$ (in terms of
amplitude) or $\vec{\omega}(k)$ (in terms of vector).

This is, besides physical observations concerning turbulence effects
in plasma, a complete justification to consider  in the rescaled
Liouville equation stochastic vector fields or potentials as it is
done in the next section.
\end{remark}

\section{The non selfconsistent stochastic approach}
\label{Sec:NSCSA}
Following the above remark, in this section, we  consider situations
where  vector field $E^\vae=-\nabla\Phi^\vae$ (or its potential) is a
random variable as  such being defined on a probability space
$(\Omega, \mathcal{F}, \mathbb{P})$   with $\mathbb{P}$ a
$\sigma$-finite probability  measure. The expectation of a random
variable $f$ is given by
\begin{equation}
\mathbb E[f]= \int_\Omega f(\omega )d \mathbb P(\omega)\,.
\end{equation}

First standard hypotheses well adapted to our presentation are
assumed on the vector field  $E^\vae$. 
\begin{itemize}

\item {\bf H1}. Stochastic average of $E^\vae$ set equal to $0$, i.e.
\begin{equation}
\forall\,(t,x), \quad \mathbb E [E^\vae (t,x)]=0\,.  \label{H11}
\end{equation}

\item {\bf H2}. Finite time decorrelation: there exists a finite positive number
  $\uptau$ such that 
\begin{equation}
  |t-s|  \ge \uptau  {\vae^2}\  \ \Longrightarrow \ \  \mathbb E [E^\vae(t,x) \otimes
    E^\vae(s,y)]=0, \quad \forall (x,y) \,. \label{H22}
\end{equation}

\item {\bf H3}.  Time and space homogeneity.  To emphasize the
  interplay between time oscillations and randomness one assumes, for
  the Fourier coefficients of the vector field (then also for the
  potential) the following  form,
\begin{equation}\label{Os0}
  E^\vae(t,k) =\fint E^\vae(t,x) e^{-ik\cdot x} dx =
  \underline{E}^\vae(t,k)e^{-i\frac{\omega_k t}{\vae^2}} =  -ik  \underline{\Phi}
  ^\vae(t,k) e^{-i\frac{\omega_k t}{\vae^2}}\,.
\end{equation}
Potentials $ \underline{\Phi}^\vae(t,k) $ are
complex random variables, while frequencies $\omega_k$ are real and
$(t,\,\vae)$-independent. We have also the following parity properties,
\begin{equation*}
  \label{oscil}
\forall k\in \mathbb Z^d, \quad \underline{\Phi}^\vae
(t,-k)=(\underline{\Phi}^\vae(t,k))^\star \quad \hbox{and}\quad
\omega_{-k}=-\omega_k\,.
\end{equation*}
Moreover, one assumes the following time and space homogeneity properties.
For any $k\in \mathbb Z^d\backslash \{0\}$, there exists a function
 $\sigma \mapsto A_k(\sigma)$ such that one has
 \begin{equation}
 \mathbb E[  \underline{\Phi}^\vae(t,k) \underline{
     \Phi}^\vae(s,k')] =  A_k\Big( \frac{t-s}{\vae^2}\Big) \delta(k+k')\,,
   \label{estimetstoc}
 \end{equation}
 with the following properties:
 \begin{equation}
   \label{estimetstoc2}
\sigma \mapsto A_k( \sigma) \mbox{ is even, } \ \quad  \forall |\sigma| >\uptau \Longrightarrow A_k(
 \sigma)=0\,, \quad \hbox{ and } \quad \sum_{k\in\mathbb Z^d} \int_\R |k|^3 |A_k(\sigma)|d\sigma <C_1\,,
 \end{equation}
 with $C_1$ being independent of $\vae$.
\end{itemize}
 
In the right-hand side of (\ref{estimetstoc} ) the term $A_k((t-s)/{\vae^2})$
stands for  the time homogeneity assumption, while the term $\delta(k+k')$
represents the hypothesis of space  homogeneity.
Observe that the functions $\sigma\mapsto A_k(\sigma)$ can
be extended by parity as functions defined on $\R$
and with  their Fourier transforms given by
\begin{equation}
  \hat{A}_k(s) =\int_\R A_k(\sigma) e^{-i s\sigma}
  d\sigma\,. \label{hatA}
\end{equation} 
As a consequence of the above hypotheses one obtains, for $\vae$ small enough,
\begin{eqnarray*}
   \mathbb D^{\vae}(t,v)&=&
  \int_0^{\frac{ t}{\vae^2}}d\sigma \, \mathbb E[
    E^\vae(t,x+ \sigma v)     \otimes  E^\vae
    (t-\vae^2\sigma,x)]\\
  &=&\sum_{k\in \mathbb Z^d\backslash\{ 0 \}} k\otimes k
  \int_0^{\frac{t}{\vae^2}} A_k(\sigma) e^{-i (\omega_k -k\cdot v)\sigma} d\sigma\,\\
  &=&\frac 12\sum_{k\in \mathbb Z^d\backslash\{ 0 \}} k\otimes k
  \int_\R A_k(\sigma) e^{-i (\omega_k -k\cdot v)\sigma} d\sigma\,.
\end{eqnarray*}
 
\subsection{Properties of the Reynolds electric stress tensor}

Properties of the Reynolds electric stress tensor $\mathbb D^\vae$
and of its limit as $\vae \rightarrow 0$ are collected in

\begin{proposition} \label{bbochner} Under assumptions
  {\bf H2} et {\bf H3} (see (\ref{H22})-(\ref{estimetstoc})),
  \begin{enumerate}
  \item The functions $s\mapsto \hat{A}_k(s)$ are real non negative,
   analytic and satisfy the estimate
    \begin{equation}\label{PW}
      \sup_{s\in \R}\sum_{k\in\mathbb Z^d}   |k|^3  |\hat{A}_k(s)|   < C_1\,.
    \end{equation}
  \item The limit of the Reynolds electric stress  tensor 
\begin{equation*}
 \mathbb D^\vae(t,v)  =\int_0^{\frac{ t}{\vae^2}}d\sigma\,
 \mathbb E [ E^\vae(t,x+ \sigma v)     \otimes    E^\vae
   (t-\vae^2\sigma,x)]
\end{equation*}
is a real non negative symmetric diffusion matrix, which is analytic in the variable $v$ and given by
\begin{equation}
  \label{LDV}
\mathbb D(v)=\overline{\mathbb D^\vae(t,v) }  = \frac 12 
\sum_{k\in \mathbb Z^d\backslash \{0 \}} k\otimes k \
\hat{A}_k(\omega_k- k\cdot v)\,.
 \end{equation}

\item  For any $\psi, \, \phi \in \mathcal D(\R^d_v) $, one has
\begin{equation}
  \label{point3}
  \overline{\int_{\R^+_t \times \R^d_v} (\nabla \phi(t,v))^T \,\mathbb D^\vae(t,v)
    \nabla_v \psi(t,v) dv dt} = -  \int_{\R^+_t \times \R^d_v}
  \phi(t,v) \nabla_v\cdot \big(\mathbb D(v) \nabla_v \psi(t,v) \big)dvdt\,.
\end{equation}
\end{enumerate}
\end{proposition}
  
\begin{proof}
  The reality of $\hat{A}_k$ follows from the parity of the
  function $\sigma \mapsto A_k(\sigma)\,$. Then for any continuous
  and compactly supported function $\R\ni s\mapsto \phi(s)$, using
  (\ref{estimetstoc}) and obvious changes of variables in time, one obtains
  \begin{equation}
  \int_\R\int_\R A_k(t-s)
    \phi(s)\phi(t)dsdt =   \vae^4\int_\R
    \int_\R \mathbb E\big[ \phi(t/\vae^2) \underline{\Phi}^\vae(t,k)
      \big(\phi(s/\vae^2) \underline{ \Phi}^\vae(s,k)\big)^\star\big]dsdt\,.
    \label{preBochner}
  \end{equation}
  Observe that the left-hand side of (\ref{preBochner}) is independent of
  $\vae$, while the right-hand side is non negative.
  As a consequence one obtains
  \begin{equation*}
  \int_\R \int_\R A_k(t-s)
  \phi(s)\phi(t)dsdt \ge 0\,,
  \end{equation*}
  and the positivity of $\hat{A}_k$ follows from the Bochner
  theorem (see \cite{Yos80}).  Eventually the fact that
  functions $\hat{A}_k(s) $ are analytic (an elementary version
  of the Paley-Wiener theorem) and satisfy estimate
  (\ref{PW}) is a direct consequence of (\ref{estimetstoc2}).  In
  the same way the rest of the proof also follows directly from
  (\ref{Os0})-(\ref{estimetstoc}).
\end{proof}
     
\subsection{Decorrelation} 
 
Assuming that the electric field is a stochastic function and
introducing the expectation in the formula (\ref{biduhamel}) one
obtains
 
\begin{equation}
  \label{Ebiduhamel}
  \begin{aligned}
&-\overline{\mathbb E\left[\fint dx\int_{\R^d_v} dv\, \phi(v)
      \nabla_v \cdot\left(\frac{E^\varepsilon f^\varepsilon}{\varepsilon} \right)\right]}\\ &=
  \overline{\mathbb E\left[\int_{\R^d_v}dv   \fint dx
      \int_0^{\frac{t}{\vae^2}}d\sigma  f^\varepsilon(
      t-\sigma\vae^2,x,v ) \nabla_v \cdot \big( E^\vae (t-\vae^2\sigma,x)
      \otimes E^\vae(t,x+ \sigma v)\nabla_v  \phi(v) \big)\right]}\,,
 \end{aligned}
\end{equation}
which leads to a "smooth" well defined diffusion matrix but which
also requires a decorrelation formula more or less of the
following type,
\begin{equation*}
\begin{aligned}
 &\mathbb E \left[\int_{\R^d_v}dv      (\nabla_v \phi(v))^T \fint dx
    \int_0^{\frac{t}{\vae^2}}d\sigma\,  E^\vae(t,x+ \sigma v)
    \otimes    E^\vae (t-\vae^2\sigma,x) \nabla_v f^\vae(
    t-\sigma\vae^2,x,v )\right]\\ &\simeq \\ &  \int_{\R^d_v}dv\,
  (\nabla_v \phi(v))^T \fint dx  \int_0^{\uptau }d\sigma\,
  \mathbb E[E^\vae(t,x+ \sigma v)     \otimes  E^\vae
    (t-\vae^2\sigma,x)]\ \mathbb E [\nabla_v f^\varepsilon(
    t-\sigma\vae^2,x,v )]\,,
\end{aligned}
\end{equation*}
and this is the object of the following lemma and proposition.
\begin{lemma}\label{lemmeproba}
{(Time decorrelation property between $f^\varepsilon$ and
  $E^\varepsilon$.)}
\label{decoref}
Assume {\bf H1} and {\bf H2}  (see (\ref{H11})-(\ref{H22})).
Suppose that
the random initial data $f_0^\varepsilon$ and the electric field
$E^\varepsilon$ are independent. Then  the operator
$E^\varepsilon (s)\cdot\nabla_v$ is independent of $f^\varepsilon(t)$ as
soon as $s\geq t+\varepsilon^2 \uptau$.   
\end{lemma}

\begin{proof} 
  From the Duhamel formula
  \begin{equation}
    f^\varepsilon(t)=  S_t^\varepsilon f_0^\varepsilon -\frac1{\vae}
    \int_0^t d\sigma\, S_{t-\sigma}^\varepsilon  E^\varepsilon(\sigma)\nabla_v
    f^\varepsilon(\sigma)\,, \label{duhamel2}
\end{equation}
  we observe that $f^\varepsilon(t)$ depends only of $f_0^\varepsilon$
  and $ E^\vae(\sigma)$  for $\sigma \leq t$. Since
  $f_0^\varepsilon$ is independent of $E^\varepsilon(t)$, $\forall
  t\in\R$, and since the electric fields $E^\varepsilon(s)$ and
  $E^\varepsilon(t)$ are independent as soon as $s > t +
  \varepsilon^2 \uptau$ (assumption {\bf H2} or  (\ref{H22})),
  Lemma~\ref{decoref} follows directly from (\ref{duhamel2}).
\end{proof}  
 
In the Duhamel formula connecting the solution from the time
$t-\vae^2\uptau$ to the time $t$,
\begin{equation}
\label{duhamel3}
f^\varepsilon(t)=  S_{\varepsilon^2\uptau}^\varepsilon
f^\varepsilon(t-\varepsilon^2\uptau) - \frac1 \vae
\int_0^{\varepsilon^2\uptau}d\sigma\, S_{\sigma}^\varepsilon E^\vae(
t-\sigma)\cdot\nabla_v f^\varepsilon(t-\sigma)\,,
\end{equation}
we insert for $f^\varepsilon(t-\sigma)$ the Duhamel formula connecting
the solution from the time $t-2\vae^2\uptau$ to the time $t-\sigma$ to
obtain
\begin{multline}
  \label{dduhamel1}
  f^\varepsilon(t)=  S_{\varepsilon^2\uptau}^\varepsilon
f^\varepsilon(t-\varepsilon^2\uptau) -\frac1{\vae}
\int_0^{\varepsilon^2\uptau}d\sigma\, S_{\sigma}^\varepsilon E^\vae(
t-\sigma)\cdot \nabla_v( S_{-\sigma}^\varepsilon
S_{2\varepsilon^2\uptau}^\varepsilon
f^\varepsilon(t-2\varepsilon^2\uptau) ) \\ +
\frac1{\vae^2}\int_0^{\varepsilon^2\uptau} d\sigma
\int_0^{2\varepsilon^2\uptau-\sigma}  ds\, S_{\sigma}^ \vae E^\vae(
t-\sigma,\cdot)\nabla_v(S_{s}^\varepsilon( E^\vae(
t-\sigma-s)\cdot\nabla_v f^\varepsilon(t-\sigma-s)))\,,
\end{multline}
 which  provides the essential tool for the needed decorrelation
 property according to the following 
\begin{proposition}\label{decorelessentiel}
Assume that the vector field $E^\vae$ satisfies Hypotheses
{\bf  H1} or (\ref{H11}) and {\bf H2} or (\ref{H22}), then for the
expectation of the Fick term one obtains
\begin{multline*}
-\nabla_v\cdot \mathbb{E}\left[\fint  dx\,
      \frac{E^\varepsilon(t) f^\varepsilon(t)}{\varepsilon} \right] \\ =
    \frac1{\vae^2} \int_0^{\varepsilon^2\uptau}d\sigma \fint
    dx\, \mathbb{E} [ E^\vae(t)\cdot\nabla_v S^\vae_\sigma
      E^\vae(t-\sigma)\cdot\nabla_v
      S^\vae_{-\sigma} ] \mathbb{E} [
      f^\varepsilon(t-2\varepsilon^2\uptau)     ] +\mathbb
    E[\mu_t^\vae]\,,
\end{multline*}
with
\begin{equation*}
 \mu_t^\varepsilon =- \frac 1
    {\vae^3}\int_0^{\varepsilon^2\uptau} d\sigma
    \int_0^{2\varepsilon^2\uptau-\sigma}  ds \fint
    dx\, E^\vae(t)\cdot\nabla_v S^\vae_\sigma
      E^\vae(t-\sigma)\cdot \nabla_v
      S_{s}^\varepsilon E^\vae(t-\sigma-s)\cdot\nabla_v
      f^\varepsilon(t-\sigma-s)]\,.
\end{equation*}
\end{proposition}
\begin{proof}
Applying operator $E^\vae(t)\cdot\nabla_v$ to (\ref{dduhamel1}),
and then applying successively the average in space and the
expectation value, we obtain
\begin{multline}
  \label{dduhamel2}
  -\nabla_v\cdot \mathbb{E}\left[\fint
    dx\,\frac{E^\varepsilon(t) f^\varepsilon(t)}{\varepsilon} \right]=
  \frac1{\vae}  \fint
  dx\,\mathbb{E}\big[E^\vae(t)\cdot\nabla_v
    S_{\varepsilon^2\uptau}^\varepsilon
    f^\varepsilon(t-\varepsilon^2\uptau)\big]\\ + \frac1{\vae^2}
  \int_0^{\varepsilon^2\uptau} d\sigma \fint  dx\,
  \mathbb{E}\left[ E^\vae(t)\cdot\nabla_v S_{\sigma}^\varepsilon
    E^\vae(t-\sigma)\cdot\nabla_v S_{-\sigma}^\varepsilon
    S_{2\varepsilon^2\uptau}^\varepsilon
    f^\varepsilon(t-2\varepsilon^2\uptau) \right] + \mathbb E[
    \mu_t^\varepsilon]\,,
\end{multline}
with
\begin{equation*}
  \label{def:eta}
\mu_t^\varepsilon= -\frac1{\vae^3}  \int_0^{\varepsilon^2\uptau}
  d\sigma \int_0^{2\varepsilon^2\uptau-\sigma}  ds \fint  dx\,
  E^\vae(t)\cdot\nabla_v S_{\sigma}^\varepsilon
    E^\vae(t-\sigma)\cdot\nabla_v S_{s}^\varepsilon
    E^\vae(t-\sigma-s)\cdot\nabla_v
    f^\varepsilon(t-\sigma-s)\,.
\end{equation*}
 Using Lemma~\ref{lemmeproba},  we obtain that $f^\varepsilon(t)$ is
 independent of  $E^\vae(s)\cdot\nabla_v$ as soon as $s\geq t+
 \varepsilon^2 \uptau$. Then, using  hypothesis ${\bf H1}$, we
 obtain 
\begin{equation*}
  \label{dduhamel3}
  \mathbb{E}\big[E^\vae(t)\cdot\nabla_v
    S_{\varepsilon^2\uptau}^\varepsilon
    f^\varepsilon(t-\varepsilon^2\uptau)\big] =
  \mathbb{E}\big[E^\vae(t)\cdot\nabla_v \big]
  S_{\varepsilon^2\uptau}^\varepsilon
  \mathbb{E}\big[f^\varepsilon(t-\varepsilon^2\uptau)\big]=0\,,
\end{equation*}
anf the first term of the right-hand side of (\ref{dduhamel2})
vanishes. Since Proposition~\ref{decoref} implies that
$E^\vae(t)\cdot\nabla_v$ and
$E^\vae(t-\sigma)\cdot\nabla_v$ are independent  of
$f^\varepsilon(t-2\varepsilon^2\uptau)$, for $0\leq \sigma
\leq\varepsilon^2\uptau$,  we obtain from (\ref{dduhamel2}),
\begin{multline*}
%  \label{dduhamel4}
  \frac{1}{\vae^2}\int_0^{\varepsilon^2\uptau} d\sigma \fint
  dx\, \mathbb{E}\left[ E^\vae(t)\cdot\nabla_v
    S_{\sigma}^\varepsilon E^\vae(t-\sigma)\cdot\nabla_v
    S_{-\sigma}^\varepsilon S_{2\varepsilon^2\uptau}^\varepsilon
    f^\varepsilon(t-2\varepsilon^2\uptau) \right]\\ =
  \frac{1}{\vae^2} \int_0^{\varepsilon^2\uptau}d\sigma \fint
  dx\,\mathbb{E}\left[E^\vae(t)\cdot\nabla_v
    S_{\sigma}^\varepsilon E^\vae(t-\sigma)\cdot\nabla_v
    S_{-\sigma}^\varepsilon\right]
  \mathbb{E}\left[S_{2\varepsilon^2\uptau}^\varepsilon
    f^\varepsilon(t-2\varepsilon^2\uptau) \right]\,.
\end{multline*}
\end{proof}

\subsection{Weak limits}

The asymptotic behavior of the error term $\mu_t^\vae$ has $\vae\rightarrow 0$ is
given by

\begin{proposition}\label{reste0}
For any $\phi \in \mathcal D(\R^+_t\times \R^d_v) $ one obtains
\begin{equation*}
 | \langle  \mu_t^\vae \,,\,\phi \rangle |\le \vae \uptau^4 C(\phi)
 \mathbb E\left[\|E\|_{L^\infty(\R_t^+; W^{2,\infty}(\torus))}^3\right]\,.
\end{equation*}
\end{proposition}  
\begin{proof}
First changing $(\sigma,s)$ into  $(\vae^2\sigma,\vae^2 s)$, with any
$\phi \in \mathcal D(\R^+_t\times \R^d_v)$,    one obtains
\begin{multline*}
  \langle\mu_t^\vae \,,\,   \phi\rangle =\int_{\R^+_t\times
    \R^d_v}dtdv \, \phi(t,v)\\\frac1{\vae^3}  \int_0^{\varepsilon^2\uptau}
  d\sigma \int_0^{2\varepsilon^2\uptau-\sigma}  ds \fint  dx\,
  E^\vae(t)\cdot\nabla_v S_{\sigma}^\varepsilon
  E^\vae(t-\sigma)\cdot \nabla_v S_{s}^\varepsilon
  E^\vae(t-\sigma-s)\cdot\nabla_v f^\varepsilon(t-\sigma-s)
  \\ \hspace{-8cm} =  \vae \int_{\R^+_t\times \R^d_v}dt dv \, \phi(t,v)
   \int_0^{ \uptau}
  d\sigma \int_0^{2 \uptau-\sigma}  ds \fint  dx\, \\
  E^\vae(t)\cdot\nabla_v S_{\vae^2\sigma}^\varepsilon
  E^\vae(t-\vae^2\sigma)\cdot \nabla_v   S_{\vae^2s}^\varepsilon
  E^\vae(t-\vae^2(\sigma+s))\cdot \nabla_v
  f^\varepsilon(t-\vae^2(\sigma+s))\,.
\end{multline*}
Then, with several integrations by part and using the fact that
${{S^\vae}_t}^*=S^\vae_{-t}$ one obtains (see \cite{BB}) 
\begin{multline}\label{ouf0}
\langle  \mu_t^\vae \,,\, \phi \rangle = -\varepsilon
\int_{\R^+}dt \int_{\R^d} dv\int_{0}^{ \uptau} d\sigma \int_{0}^{2
  \uptau-\sigma} ds \fint dx\,\\
  f^\varepsilon(t-\varepsilon^2(\sigma+s)) E^\vae
  (t-\varepsilon^2(\sigma+s))\cdot \nabla_v(
  S_{-\varepsilon^2s}^\varepsilon E^\vae
  (t-\varepsilon^2\sigma)\cdot\nabla_v
  (S_{-\vae^2\sigma}^\vae E^\vae(t)\cdot\nabla_v\phi))\,.
\end{multline}
In the last line of (\ref{ouf0}) appears the term 
$$
E^\vae (t-\varepsilon^2(\sigma+s)) \cdot\nabla_v(
S_{-\varepsilon^2s}^\varepsilon E^\vae
(t-\varepsilon^2\sigma)\cdot\nabla_v(
S_{-\varepsilon^2\sigma}^\varepsilon E^\vae(t) \cdot \nabla_v
\phi))\,,
$$
which contains at most second order derivatives with
respect to $v$ of expressions of the form $E^\vae(s, x+\tilde \sigma
v)$. With $\uptau$ finite, $x\in \torus$  and the with the
introduction of a test function $\phi\in \mathcal{D}(\R^+_t \times
\R^d_v)$,  the support of the integrand is bounded in
$\R^+_t\times\torus\times\R^d_v\,.$ Then,  with a crude estimate (that
could be improved) one obtains
 \begin{equation}\label{reste}
 | \langle  \mu_t^\vae \,,\,\phi \rangle |\le \vae \tau^4 C(\phi)
\|E^\vae\|_{L^\infty(\R_t^+; W^{2,\infty}(\torus))}^3\,.
 \end{equation}
 Finally, taking the expectation of (\ref{reste}) one concludes
 the proof of Lemma~\ref{reste0}.
\end{proof}

The diffusion limit is given by

\begin{proposition}\label{dual}
 For any $\phi \in\mathcal D(\R^+_t\times \R^d_v) $, one obtains
\begin{equation}
\label{ouf2}
 \overline{\int_{\R^+_t \times \R^d_v} dt dv \,\phi \nabla_v\cdot
   \mathbb{E}\left[\fint  dx\, \frac{E^\varepsilon(t)
       f^\varepsilon(t)}{\varepsilon} \right]} = -  \int_{\R^+_t
   \times \R^d_v}  dtdv\, \overline{\fe}(t,v)  \nabla_v\cdot \big( \mathbb D(v) \nabla
 \phi(t,v)\big)\,.
 \end{equation}
\end{proposition}
 
\begin{proof}
Knowing already from Proposition~\ref{reste0} that the reminder
$\mu^\vae_t$ goes to $0$ in $\mathcal D '(\R_t^+\times\R^d_v) $, one obtains
\begin{equation}
  \label{ouf3}
 \int_{\R^+_t \times \R^d_v} dt dv \,\phi \nabla_v\cdot
   \mathbb{E}\left[\fint  dx\, \frac{E^\varepsilon(t)
       f^\varepsilon(t)}{\varepsilon} \right] = I^\vae,
\end{equation}
with
\begin{equation*}
\begin{aligned}
I^\vae:=
- \int_{\R^+_t } dt \int_{\R^d_v} dv  \frac{1}{\vae^2}\int_0^{\vae^2\uptau}d\sigma \fint dx\, \phi(t,v)
\mathbb{E}&\left[E^\vae(t)\cdot\nabla_v
    S_{\sigma}^\varepsilon E^\vae(t-\sigma)\cdot\nabla_v
    S_{-\sigma}^\varepsilon\right]\\
    &
  \mathbb{E}\left[S_{2\varepsilon^2\uptau}^\varepsilon
    f^\varepsilon(t-2\varepsilon^2\uptau) \right]\,.
\end{aligned}\end{equation*}  
After expanding the integrand of $I^\vae$, using an integration by parts in $v$,
and changing $\sigma$ into $\vae^2 \sigma$, one obtains
\begin{multline*}
I^\vae:=\int_{\R^+_t } dt \int_{\R^d_v} dv \int_0^{\uptau}d\sigma \fint dx\, \nabla_v\phi^T\,
\mathbb{E}[E^\vae(t,x)\otimes
  E^\vae(t-\vae^2\sigma,x-v\sigma )]
\\ \mathbb{E}[
  (\sigma-2\uptau)(\nabla_xf)(t-2\vae^2\uptau,x-2v\uptau,v) + (\nabla_vf)(t-2\vae^2\uptau,x-2v\uptau,v)
]\,.  
\end{multline*}  
Using the change of variables $(t,\,x,\,v) \rightarrow (t'=t-2\vae^2\uptau,\,x'=x-2v\uptau,\,v'=v)$
and integration by parts in $(x,v)$ one obtains
\begin{equation*}
  I^\vae:=- \frac{1}{(2\pi)^d}\int_{\R_t } dt \int_{\R^d_v} dv \int_0^{\uptau}d\sigma\int_{\R_x^d} dx\,
  \mathbb{E}[f^\vae(t,x,v)] \Psi^\vae(t,\sigma, x, v)\,,
\end{equation*}  
with
\begin{equation}
\begin{aligned}
 & \Psi^\vae(t,\sigma, x, v)= 
  \mathbbm{1}_{[-2\vae^2\uptau,\,+\infty[}(t) \mathbbm{1}_{\{\torus -2v\uptau\}}(x)  
      (\nabla_v\cdot\, + \,(\sigma-2\uptau)\nabla_x \cdot\,) \\ 
     & \mathbb{E}[E^\vae(t-\vae^2(\sigma-2\uptau),x-v(\sigma-2\uptau))\otimes
  E^\vae(t+2\vae^2\uptau,x+2v\uptau )]\nabla_v\phi(t+2\vae^2\uptau,v)
  \,.
\end{aligned}
\end{equation}
The domain of integration of the integral $ I^\vae$ is a compact set $K$
of $\,\R_t\times\R^d_v\times[0,\uptau]_\sigma\times\torus$. We already know
that $\mathbb{E}[f^\vae(t,x,v)]\rightharpoonup \overline{f^\vae}(t,v)$ in $L^\infty(K)$
weak$-\star$. It remains to show the strong convergence in $L^1(K)$ of
$\Psi^\vae$ to a suitable cluster point $\overline{\Psi^\vae}$. For this, using (\ref{Os0}),
one obtains
\begin{multline*}
  \Psi^\vae(t,\sigma, x, v)=\mathbbm{1}_{[-2\vae^2\uptau,\,+\infty[}(t) \mathbbm{1}_{\{\torus -2v\uptau\}}(x) \
      (\nabla_v\cdot\, + \,(\sigma-2\uptau)\nabla_x \cdot\,) \\ 
      \sum_{k,k'\in \Z^d} -k\otimes k' \, e^{i(k+k')\cdot x}  e^{i2(k+k')\cdot v\uptau } e^{-i 2(\omega_k+\omega_{k'})\uptau }
      e^{-i(\omega_k+\omega_{k'})\frac{t}{\vae^2} } e^{-i(\omega_k-k\cdot v) \sigma}
     \\ \mathbb{E}[\underline{\Phi}^\vae(t-\vae^2(\sigma-2\uptau),k)
        \underline{\Phi}^\vae(t+2\vae^2\uptau,k')]\nabla_v\phi(t+2\vae^2\uptau,v)\,.
\end{multline*}
Without space homogeneity, we observe that
the term $\exp(-i(\omega_k+\omega_{k'}){t}/{\vae^2})$ does not converge
pointwise almost everywhere in time, which prevents strong convergence in $L^1(K)$
of the function $\Psi^\vae$. By constrast, using the spatio-temporal homogeneity property
(\ref{estimetstoc} ) one obtains
\begin{multline*}
  \Psi^\vae(t,\sigma, x, v)=\mathbbm{1}_{[-2\vae^2\uptau,\,+\infty[}(t) \mathbbm{1}_{\{\torus -2v\uptau\}}(x) \
      (\nabla_v\cdot\, + \,(\sigma-2\uptau)\nabla_x \cdot\,) \\ 
      \sum_{k\in \Z^d} A_k(\sigma) e^{-i(\omega_k-k\cdot v) \sigma} k\otimes k \nabla_v\phi(t+2\vae^2\uptau,v)\,.
\end{multline*}
Using regularity properties (\ref{estimetstoc2} )and Lebesgue dominated convergence theorem,
one obtains that  $\Psi^\vae$ converges in $L^1(K)$ strong towards
the cluster point $\overline{\Psi^\vae}$, which is defined by
\begin{equation*}
  \overline{\Psi^\vae(t,\sigma, x, v)}=\mathbbm{1}_{\R^+}(t) \mathbbm{1}_{\{\torus -2v\uptau\}}(x) \
      (\nabla_v\cdot\, + \,(\sigma-2\uptau)\nabla_x \cdot\,)
      \sum_{k\in \Z^d} A_k(\sigma) e^{-i(\omega_k-k\cdot v) \sigma} k\otimes k \nabla_v\phi(t,v)\,.
\end{equation*}
Using properties  (\ref{estimetstoc2}) for $A_k$, and passing to the limit $\vae \rightarrow 0$ in $I^\vae$,
one obtains
\begin{equation*}
  \overline{I^\vae}=-\int_{\R^+_t } dt \int_{\R^d_v} dv\,
  \overline{f^\vae}(t,v) \nabla_v \cdot \Big(\frac 12
  \sum_{k\in \Z^d}  k\otimes k\, \int_{\R} d\sigma\, A_k(\sigma) e^{-i(\omega_k-k\cdot v) \sigma} \nabla_v\phi(t,v)
  \Big)\,.
\end{equation*}
Using this last equation, definitions (\ref{hatA})  and (\ref{LDV}), and passing
to the limit $\vae \rightarrow 0$ in (\ref{ouf3}), ones obtains  (\ref{ouf2}),
which ends the proof of the Proposition~\ref{dual}.
 \end{proof}
 
\subsection{The basic stochastic theorem}
From the above derivation one  deduces 
\begin{theorem}
  Let $\{E^\vae(t,x; \omega)\}_{\omega \in \Omega}=\{-\nabla \Phi(t,x;\omega)\}_{\omega \in \Omega}$
  be a family of stochastic (with
  respect to the random variable $\omega\in \Omega$)
  gradient vector fields. Assume that such vector fields
  satisfy the $\vae$-independent local in time regularity hypothesis
\begin{equation*}
\forall \vae>0 \ \ \hbox{ and } \ \  \forall T>0, \ \  \sup_{0<t<T}\|
E^\vae (t)\|_{W^{2,\infty}(\torus)} \le  C(T)\,,
\end{equation*}
and the  following detailed ergodicity hypotheses. With
\begin{equation*}
\forall k\in \mathbb Z^d\backslash\{ 0\}, \quad
E^\vae(t,k) =\fint E^\vae(t,x) e^{-ik\cdot x} dx =
\underline{E}^\vae(t,k)e^{-i\frac{\omega_k t}{\vae^2}} =  -ik  \underline{\Phi}
  ^\vae(t,k) e^{-i\frac{\omega_k t}{\vae^2}}\,, 
\end{equation*}
there exist a constant $\uptau\in (0,\,+\infty)$  and functions $\sigma\mapsto A_k(\sigma)$, 
 $k\in \mathbb  Z^d\backslash \{0\}$, 
such that   
\begin{equation*} \label{repeat}
  \begin{aligned}
    &\mathbb E[  \underline{\Phi}^\vae(t,k) \underline{
        \Phi}^\vae(s,k')] =  A_k\Big( \frac{t-s}{\vae^2}\Big) \delta(k+k')\,,
    \\ &\sigma \mapsto A_k( \sigma) \mbox{ is even, }  \quad
   |\sigma| >\uptau\Longrightarrow A_k(\sigma)=0\,, \quad \hbox{and } \quad 
   \sum_{k\in\mathbb Z^d} \int_\R |k|^3 |A_k(\sigma)|d\sigma
   <C_1\,, 
 \end{aligned}
\end{equation*}
with  $C_1 $ being independent of $\vae$. 
 
Then, 
\begin{enumerate}
\item For all $k\in \mathbb  Z^d\backslash \{0\}$ the Fourier transform of the function
  $\sigma\mapsto A_k(\sigma)$ is non negative and the bounded diffusion matrix
\begin{equation}\label{diffusionmatrix}
\mathbb D(v) =  \frac 12 \sum_{k\in \mathbb Z^d\backslash\{ 0\}} k\otimes
k  \ \hat{A}_k(\omega_k-k\cdot v)  
\end{equation}
is symmetric non negative and analytic in the variable $v$.
\item Define by $V\subset L^2(\R^d_v)$ the closure of the space of
  functions $\phi\in\mathcal D(\R^d_v)$ for the norm
\begin{equation*}
\|\phi\|_V^2= \|\phi\|^2_{L^2(\R^d_v)} +\int_{\R^d_v}
\big(\mathbb D(v)\nabla_v\phi \,,\, \nabla_v \phi\big)\,.
\end{equation*}
Then for any $f_0(v)\in L^2(\R^d_v)$ there exists a unique
solution of the following problem: find
\begin{equation}
 f (t,v)  \in \mathscr{C}\big(\R^+_t;\, L^2(\R^d_v)\big)\cap L^2 \big(\R^+_t;\, V\big) \quad
\hbox{ with }\quad   f (0,v)= f_0(v)\,, \label{reg}
\end{equation}
such that $f$ is the solution  (in the sense of $\mathcal D'(\R^+_t\times \R^d_v)$)  of the diffusion equation 
\begin{equation}
\del_t f -\nabla_v \cdot (\mathbb D(v) \nabla_v f )=0\,. \label{dist}
\end{equation}
\item Since $v\mapsto \mathbb D(v)$ is regular (analytic) the time
  derivative of the solution of (\ref{dist}) is well defined (say in
  $L^\infty(\R^+_t ; \, H^{-2}(\R^d_v))$), hence the initial condition
  $f(0,v)=f_0(v)$ is well defined and with such initial data
  this solution coincides with the unique solution of the problem
   (\ref{reg})-(\ref{dist})
\item Any cluster point $\overline{\fe}$, in the
  $L_{\rm loc}^\infty(\R^+_t \times \R^d_v)$ weak$-\star$ topology, of the family
  $\mathbb E[\fe]$ with $\fe$ solution of the stochastic Liouville equation  
\begin{equation*}
\vae^2\del_t\fe + v\cdot \nabla_x \fe+ \vae E^\vae \nabla_x
\fe=0\,, \quad \quad \fe(0,x,v)=f_0(x,v)\,,
\end{equation*}
is a function $\overline{\fe}(t,v) $ independent of $x$ and solution
of (\ref{dist}) with initial data given by
\begin{equation*}
\overline{\fe} (0,x)=\fint f_0(x,v) dx\,.
\end{equation*}
\item Eventually by a uniqueness argument it is not a subquence of
  $\mathbb E[\fe]$ but the whole family that converges to the solution
  of the diffusion equation.
\end{enumerate}

\end{theorem}
\begin{proof}
The point 1 follows directly from the points 1 and 2 of the
Proposition~\ref{bbochner}. The point 2 is a classical result of
variational theory (see \cite{Lions}). The purpose of the point 3
is to prove the regularity (\ref{reg}) for solutions in the sense of
distribution. This is easily done by considering standard regularizations of $f$
in velocity such that ${f}_\epsilon =\varrho_\epsilon \mathop{\ast}_{v} f\, $.
For the point 4, the limit equation results from relation (\ref{ouf2})
of Proposition~\ref{dual}. Then, to prove the time continuity in
$H^{-2}(\R^d_v)$ one considers the  equation
\begin{equation*}
\del_t \mathbb E \left [\fint dx\, \fe (t)\right] +\nabla_v\cdot
\mathbb{E}\left[\fint  dx\, \frac{E^\varepsilon(t)
    f^\varepsilon(t)}{\varepsilon} \right]=0\,,
\end{equation*}
and from estimate (\ref{reste}) and equation (\ref{ouf3}) one
deduces that 
$$ \del_t \mathbb E \left[\fint dx\,  \fe (t)\right]
$$ is uniformly bounded in $L^\infty(\R^+_t; H^{-2}(\R^d_v))$ and
this gives a uniform estimate on the time continuity, which is enough
to complete the proof of the point 5.
\end{proof}
 
\section{Returning to the Vlasov--Poisson equations}
\subsection{Classical stability results for nonlinear and linearized Vlasov--Poisson equations}
\label{sec:CSRLVP}
As above, solutions of the Vlasov--Poisson equations  are considered on
$\R^+_t\times\torus\times\R^d_v$ and, before any rescaling, involve a
probability density $f$, which is solution of the  Liouville equation,
\begin{equation*}
\del_t f +v\cdot \nabla_x f + E\cdot \nabla_v f =0\,.\label{VP1}
\end{equation*}
The electric field $E$ is given  by a selfconsistent potential, given in
terms of the density $\rho$ by
\begin{equation*}
-\Delta \Phi (t,x) =\rho(t,x) =\int_{\R^d_v} f(t,x,v)dv -1\,, \,\quad
E(t,x) = -\nabla \Phi(t,x)\,.\label{VP2}
\end{equation*}
Existence, uniqueness and persistence of regularity (with regular
initial data) are classical (see \cite{GLA96}) and in particular for
propagation of analyticity see \cite{BEN86}). However these
regularity properties may not be uniform with respect to time and 
scaling parameters.
Moreover,  as Fourier transform with respect to the velocity and Laplace transform
with respect to time are used. They are
denoted as follows,
\begin{equation}
\begin{aligned}
&\mathcal L h(\lambda ,v)  =\int_0^{\infty} e^{-\lambda s}
  h(s,v)ds\,,\\ & h(k,v)=\fint h(x,v)e^{-ik\cdot x} dx\,,
  \\ &\mathcal F_v G(\xi)=\int_{\R^d} G(v) e^{-iv\cdot \xi}
  dv\,.
\end{aligned}
\label{def:LFT}
\end{equation}
 Properties which are independent of time   are recalled below. With
 $f(0, x,v)\in L^\infty(\torus\times\R^d_v)$ being a positive density
 of mass  $1\,$ one has
\begin{equation*}
\begin{aligned}
 &\forall t\in \R^+\,,  \quad  f(t,x,v) \ge 0\,,  \quad 1\le p <
  \infty\,, \quad \int_{\torus \times\R^d_v} |f(t,x,v)|^p dvdx =
  \int_{\torus \times\R^d_v} |f(t,x,v)|^p dv dx,\\ &\forall t\in \R^+\,,
  \quad \frac{d}{dt}   \bigg( \fint \int_{\R^d_v} \frac{|v|^2}{2}
  f(t,x,v)dvdx + \frac 1 2  \fint  |E(t,x)|^2dx \bigg)=0\,.
\end{aligned}
\end{equation*}
With the above estimates one deduces also the following time uniform
regularity:
\begin{equation*}
\hbox{In any dimension } d\,,   \, \quad     \label{eqn:II2} \| E\|_{L^\infty(\R_t^+;W^{1,1+2/d}(\mathbb{T}^d))} \leq c_0 <\infty, 
\end{equation*}
\subsection{Spectral properties of linearized Vlasov--Poisson equations}\label{spectral}
Below, to use in short time quasilinear approximation some very
classical facts are recorded.
Any density $v\mapsto G(v) \ge 0 $ with $\int_{\R^d_v}G(v)dv =1 $ is a
stationary solution of the Vlasov--Poisson equations (with $E(t,x)\equiv0$).
Hence, for $f(t,x,v)=G(v)+\vae h(t,x,v)\,,$ the $\vae$ first-order
term is solution of the following linearized equation
  \begin{equation}
    \begin{aligned}
      &\del_t h +v\cdot\nabla_x h + E[h]\cdot\nabla_v G =0\,, \\ & \fint
      dx\int_{\R^d_v} h(t,x,v) dv=0\,, \quad
      E[h](t,x)=\nabla\Delta^{-1}\int_{\R^d_v} h(t,x,v)dv\,.
    \end{aligned}
    \label{linear20}
  \end{equation}
  By standard perturbation (See \cite{KA66}) the operator
  $\mathcal T_G: h\rightarrow \mathcal T_G h =-(v\cdot \nabla_x h + E[h]\cdot
\nabla_v G)$ is the generator of a strongly continuous  group $e^{t
  \mathcal T_G}$ (particularly in $L^2(\torus\times \R^d_v)$).  Next
one observes that in this space, the map  $h\mapsto E[h]\cdot\nabla_v G\,$ is
a compact operator. Hence, with the Duhamel formula,
\begin{equation*}
e^{t \mathcal T_G} h_0=  S_t h(0)   - \int_0^t d\sigma\, S_{t-\sigma} 
E[h](\sigma)\cdot\nabla_v G(v) \,,
\end{equation*}
one deduces that
for any $\alpha>1$, the spectra of $e^{t \mathcal{T}_G}$, contained in
the region $\{\mu\in \C \, | \, |\mu|\ge \alpha>1\}$,   is a finite sum of eigenvalues
with finite multiplicity, which are the images of the poles of the
resolvent of the generator $-\mathcal T_G$,  i.e. complex numbers
$\lambda_m\,,\,\,  \Re \lambda_m>0\, $, such that the equation
\begin{equation*}
\lambda_m h + v\cdot \nabla_x h+  E[h] \cdot\nabla_v G=0\,,  \quad h
\in L^2(\torus\times \R^d_v)\,,
\end{equation*}
has a non trivial solution. Using Fourier series on $\torus$,
this  means that there exists at least one $h(k_m,v)  \in L^2(\R^d_v)$
(may be more than one if the multiplicity of $\lambda_m$ is greater
than one) such that, taking in account the relation between the
Fourier components of the electric field and of the density, namely
\begin{equation*}
E[h](k_m)=  -\frac{ik_m}{|k_m|^2} \rho[h](k_m)\,,
\end{equation*}
one has the "dispersion  equation",
\begin{equation}
 1- \int_{\R^d_v}\frac {i k_m}{|k_m|^2}\cdot\frac{\nabla_v
   G(v)}{\lambda_m + ik_m\cdot v} dv =0 \label{Penrose1}\,.
\end{equation}

Since $G(v)$ is real, one observes that if $(\lambda_m  ,k_m)$ is a
solution, then the same  same is true for   $((\lambda_m)^\ast,- k_m)$
and that in this case one obtains for the Fourier component of
$h\,,$
\begin{equation*}
h_{ \lambda_m} (k_m, v)=-\frac 1{(\lambda_m + ik_m\cdot v)} E_{\lambda_m}(k_m)\cdot \nabla_v G(v)\,.
\end{equation*}
As a consequence (assuming for sake of simplicity   that
$\lambda_m=\gamma_m + i\omega_m$ is a simple root of the analytic
equation (\ref{Penrose1}))  one observes that  $h_m(t,x,v)$,  
the solution of the equation (\ref{linear20}),
and the electric field
\[
E_m(t,x)=e^{i k_m x +\lambda_m t}E_{\lambda_m}(k_m)
\]
are binded by the
relation, for any time $t$,
\begin{equation*}
h_{m}(t, x,v) =-\frac {e^{i k_m x +\lambda_m t}}  {(\lambda_m  + i k_m\cdot v) } E_{\lambda_m}(k_m)\cdot\nabla_v
 G
 \,.\label{binded}
 \end{equation*} 
On the other hand for any $\lambda_m\,,$ introduce the Kato projector 
(see \cite{KA66} page 178), on the eigenspace corresponding to the
eigenvalue $\lambda_m$, defined by
\begin{equation*}
 P_m h_0(x,v) =\frac 1{2i\pi} \int_{\Gamma_m}(\lambda I -\mathcal
 T_G)^{-1} h_0 d\lambda\,,
\end{equation*}
where $\Gamma_m$ is a ``small" oriented contour of the complex half
plane $\Re\lambda>0$ containing only  $\lambda_m $ in its interior.
Since   for any $\delta>0 $ there is a finite number of eigenvalues in
the region
\begin{equation*}
0<\delta\le \lambda\le \Lambda =\sup \Re \lambda_m\,,
\end{equation*}
one obtains, (assuming that these eigenvalues are simple) for any real
density $h(t,x,v) $ solution of the linearized equation
(\ref{linear20}), the following asymptotic expansion,
\begin{equation}\label{Dunford-Kato}
h(t,x,v)=  \sum_{0< \delta \le\Re\gamma_m\le \Lambda}   e^{(\gamma_m
  + i\omega_m)t}e^{ik_m\cdot  x}  P_m h_0+\mathcal{O}(e^{\delta  t})\,.
\end{equation}
Using an explicit Laplace transform (which is the forerunner of the
Dunford--Kato calculus used above) one tries to move the integration
contour to the left-half plane, i.e. on the interval $\lambda +i \sigma$
with $\Re\lambda<0 $ and $-\infty<\sigma<+\infty$. In doing so, one
generates a singularity at the point $\omega=- k\cdot v$ and by the
Plemelj formula a term of the form $i\pi f(-ik \cdot v)$ which makes
sense only under the hypothesis that the function $f$ is analytic 
(see \cite{KT73} chapter 8).  These are  standard methods in linear
scattering theory (see \cite{KA66, LP67}) which in the absence of more specific
information are  based on analyticity hypotheses and convergence in
ultra distributions (dual of analytic functions). These tools were introduced by Case
and Zweifel \cite{CAS59, CZ67}, extended by Sebasti\~ao e
Silva \cite{SEB67}, and later systematized by Degond \cite{DE86}.

\subsection{Remark about the Landau damping. Comparison with the
  behavior of the vector field in the rescaled equation}  \label{bbd}
In 1946, Landau \cite{LAN46} observed that in the absence of unstable modes
(no solutions $\lambda$ of the equation  (\ref{Penrose1}) with $\Re \lambda >0)$
the electric field goes exponentially fast to zero as $t\rightarrow \infty\,.$ 
Following the recent version of Grenier, Nguyen and Rodniansky \cite{
  GNR20}, we assume that the profile  $v\mapsto G(v)$ and the initial
data $(x,v) \mapsto h_0(x,v)$ are analytic functions.
Using notation (\ref{def:LFT}), Laplace transform with  respect to time, Fourier
transform with respect to $x$ and  Fourier transform with respect to
$v$ are used.  They are denoted as in Section~\ref{sec:CSRLVP}.
We first focus  on the behavior of the density $\rho[h]$, of the solution $h$ of
the linearized equation (\ref{linear20}), given by
$$ \rho[h](t,x)= \int_{\R^d_v} h(t,x,v) dv\,.
$$

Using Fourier--Laplace transformations  we obtain, for
$\Re\lambda>0\,,$  the relation,
\begin{equation*}
\begin{aligned}
 (1 + K_G(\lambda,k)) {\mathcal L} \rho[h](\lambda,k)= \int_{\R^d_v}
  \frac{   h_0(k,v)}{\lambda +ik\cdot v}dv  \quad \hbox{with} \quad
  K_G(\lambda,k)=-\int_{\R^d_v}\frac {i k }{|k |^2}\cdot\frac{\nabla_v
    G(v)}{\lambda  + ik \cdot v}  dv\,.
\end{aligned}
\end{equation*}
Then, following \cite{MV10},  one observes  that
\begin{equation*}
K_G(\lambda,k)=\int_0^\infty e^{-\lambda s}s (\mathcal F_{v   }
G)(ks)ds\quad \hbox{and} \quad\int_{\R^d_v} \frac{\hat
  h_0(k,v)}{\lambda +ik\cdot v}dv=\int_0^\infty  e^{-\lambda
  s}(\mathcal F_{v } h_0)(k,ks)ds\,.
\end{equation*}
Therefore, as in \cite{MV10}, one proposes  a stronger criterion for  the absence of
solution $\lambda_m$ with $\Re\lambda_m>0$ for the equation
(\ref{Penrose1}), which is
\begin{equation*}
\exists \kappa_0 >0, \mbox{ such that } \quad \inf_{k\in \mathbb   Z^d ,\Re \lambda >0}\left|1+   \int_0^\infty e^{-\lambda
  s}s \mathcal F_v  G(ks)ds\right| \ge \kappa_0>0 \,.\label{superpenrose}
\end{equation*}
 Then, we  obtain for $\Re \lambda>0$ the solution
\begin{equation*}
\mathcal L \rho[h](\lambda ,k) = \frac{1}{(1+ K_G(\lambda,k))}
\int_0^\infty e^{-s\lambda} (\mathcal F_v)h_0(k,ks)ds\,. \label{laplacero}
\end{equation*} 
With the hypothesis of analyticity, by the Paley--Wiener Theorem,
there exist   $C$ and $\theta_0$ such that one has:
\begin{equation*}
|\mathcal F_v  G(ks)| \le C e^{-\theta_0|k|s}\,, \quad \hbox{ and } \quad 
|\mathcal F_v h_0(k,ks)| \le C e^{-\theta_0|k|}\,.
\end{equation*}
As a consequence, using (\ref{superpenrose}) for any  $k\in
\mathbb{Z^d}\backslash \{0\}\,,$ the functions 
\begin{equation*}
K_G(\lambda,k)=\int_0^\infty e^{-\lambda s}s \hat G(ks)ds \quad
 \hbox{and} \quad S(\lambda,k)=\int_0^\infty e^{-s\lambda}  \mathcal
F_v h_0(k,ks) ds\,,
\end{equation*}
can be extended (for the density $\rho[h]$, the integration with
respect to $v$  leads to an extension behind the imaginary axis
without extra singularity) as  analytic functions in the region $\Re
\lambda >-\theta_0|k|\,.$ Eventually one obtains (see \cite{GNR20})
that there exists $\theta_1 >0$ such that
 \begin{equation}
 \hbox{for} \quad \Re\lambda\ge -\theta_1|k|,  \quad |\mathcal L
 \rho[h](\lambda,k)| \le \frac{C_1}{1+|k|^2 + |\Im \lambda
   |^2}\,.\label{toangrenier}
 \end{equation} 
This gives the exponential decay for $\rho[h] $ and $E[h]$ when $G$ is
analytic. This  is sufficient for the present discussion (extension to
an initial data  $f_0$  belonging only to a Gevrey  space with Gevrey index
$\gamma\in ( \frac13,1]$ uses the presence of the term $|k|^2$ in
(\ref{toangrenier})). Of course,  the nonlinearity requires  more
sophisticated analysis in particular in the interaction of modes,
which  is the classical problem of the echoes. Details can be found in
\cite{GNR20} where the following theorem is obtained. 

\begin{theorem}\label{GNRth}
Assume for  the initial data, 
\begin{equation}
f_0(x,v) = G(v) + \vae h_0(x,v),
\end{equation}
 that  $G(v) $ and $h_0(x,v)$ are analytic, while the basic profile
 $G(v)$ satisfies the stability estimate (\ref{superpenrose}). Then,
 for $\vae$ small enough, the corresponding solution exhibits the
 Landau damping effect  i.e. as $ t\rightarrow \infty$, the electric field
 $E(t,x)$ goes exponentially fast to $0\,$.
\end{theorem}
As observed in the introduction, the ergodicity of  the torus $\torus$
implies that for any $0<T<\infty$ the solution $\fe(t,x,v)$, of the rescaled
equation,  converges in  $L^\infty([0,T]\times\torus\times\R^d_v)$
weak$-\star$ to an $x$-independent function. Hence, with the
Poisson equation the electric field converges also in
$L^\infty(0,T;L^2(\torus))$ weak$-\star$ to $0\,$. Therefore, this property would
justify the term "baby Landau damping". In the above situation genuine
Landau damping would correspond to strong convergence
(for any $\delta>0$ and $\delta<T<\infty$) in $L^\infty(\delta,T;L^2(\torus))\,$.
Strong convergence will be the counterpart, in the present situation, of
Theorem~\ref{GNRth}. In fact one has

\begin{theorem}\label{babydamping}
For solutions of the rescaled Vlasov--Poisson equations,
\begin{equation}
  \label{rvpe}
\begin{aligned}
&\vae^2\del_t \fe + v\cdot \nabla_x \fe + \vae E^\vae\cdot\nabla_v \fe
  =0\,, & \nabla_v\cdot  E^\vae =\rho(t,x) = \int_{\R^d_v}f(t,x,v)dv
  -1\,,
\end{aligned}
\end{equation}
near an equilibrium
\begin{equation*}
f_0(x,v) = G(v) + h_0(x,v)\,,
\end{equation*}
with analytic data and profile satisfying the stability estimate
(\ref{superpenrose}),  and no restriction on the ``size" of the
analytic perturbation $h_0$,  on $\, 0<\delta<T<\infty\, $ the  electric field
$E^\vae(t,x)$ converges  exponentially fast to $0\,$  in
$L^\infty(\delta,T ;L^2(\torus))$ as $\vae \rightarrow 0$.
\end{theorem}
\begin{proof}
 For  strong convergence one follows \cite{GNR20} .  First introduce
 the function $F^\vae (\tau, x,v)$ solution of the equation
\begin{equation*}
\del_\tau F^\vae +v \cdot\nabla_x F^\vae +  E^\vae[\vae \rho^\vae]\cdot
\nabla_v F^\vae\,, \label{GNR1}
\end{equation*}
with 
\begin{equation*}
\fe(t,x,v) = F^\vae ({t}/{\vae^2},x,v) \quad \hbox{and}
\quad\vae E^\vae[\rho^\vae]  = E^\vae[\vae\rho^\vae]\,.\label{trivial}
\end{equation*}
For $\vae\le \vae_0$ small enough, apply Theorem~\ref{GNRth} to
the solution $(F_\vae(\tau,x,v), E_\vae(\tau,x))$ of 
\begin{equation*}
\begin{aligned}
&\del_\tau F + v\cdot\nabla_x F + E(\tau) \cdot   \nabla_v F
  =0\,, \\ & F_\vae(0, x,v) =
  G(v) +\vae h_0(x,v)\,,
\end{aligned}
\end{equation*}
which coincide with the solution $(f^\vae(t,x,v), E^\vae(t,x))$ of (\ref{rvpe})
through the relation $(f^\vae(t,x,v),E^\vae(t,x,v))=(F_\vae(t/\vae^2,x,v),E_\vae( t/\vae^2,x))$.
 \end{proof} 
 
  \subsection{Road map for the short time quasilinear approximation}\label{RMQA}
  
  In this section, we  present a prospective method to proof the
  validity of a short time quasilinear approximation in the presence
  of unstable eigenvalues. Observe that for any time $t>0$  with,
$$
\fe(t,x,v) = G(t,v) + \vae h (t,x,v) \,, \quad
\fint h(t,x,v)dx=0\,,
$$
the Vlasov--Poisson equations are equivalent to the system,
 \begin{equation}\label{qasilinear}
 \begin{aligned}
 &\del_t G +\vae^2 \nabla_v\cdot \Big(\fint E[h] h dx\Big)
   =0\,,\quad E[h]=\nabla  \Delta^{-1} \int_{\R^d_v} h(t,x,v)dv\,, 
   \\ &\del_t h +v\cdot\nabla_x h + E	[h]\cdot\nabla_v G= - \vae
     \nabla_v \cdot\Big(E[h]   h-\fint E[h]  h dx\Big) \,.
 \end{aligned}
 \end{equation} 
  Next, assume the existence of  a simple non degenerated root
  $(\Re\lambda(0)>0, k) $ of the dispersion equation 
 \begin{equation}
 1 - \frac{1}{|k |^2} \int_{\R^d_v} \frac{ik \cdot \nabla_v G(0 ,v)}{
   \lambda   +ik \cdot v} dv=0\,,\label{penrose2}
 \end{equation}
{with $G(0,v)=G_0(v)$}, and consider solutions  $\fe(t,x,v)$ of the
 Vlasov equation with initial data
 \begin{equation*}
 \fe(0,x,v)=G_0(v)+\vae h_0(x,v)=G_0(v) +\vae \frac{E(0,k) \cdot \nabla_v
   G_0(v)}{ \lambda  +ik \cdot v} e^{ik\cdot
   x}\,. \label{slaved2}
 \end{equation*}
 Assuming that $ \fe(0,x,v)$ is  analytic, one observes (as proven by
 Benachour in \cite{BEN86}) that the corresponding solution of the
 Vlasov equation is also analytic. Hence (see \cite{KA66} Chapter~2,
 Section~1) the root can be extended on a finite time interval as
 simple solution of the equation
 \begin{equation*}
 1 - \frac{1}{|k |^2} \int_{\R^d_v} \frac{ik \cdot \nabla_v G(t ,v)}{
   \lambda(t)   +ik \cdot v} dv=0\,, \label{penroset}
 \end{equation*}
 and then one  introduces the approximate solution
 \begin{equation}
 \tilde h(t,x,v) = \frac{E(0,k) \cdot \nabla_v G(t,v)}{ \lambda(t)  +ik
   \cdot v}  e^{ {\int_0^t ds\lambda(s)} +ik\cdot x}\,. \label{ansatz}
 \end{equation} 
 The function  $\tilde h$ will be used to construct an approximate diffusion
 $\widetilde{\mathbb D}^\vae(v)$ such that, for short time one has
 \begin{equation*}
   \del_tG(t,v)-\nabla_v\cdot \big(\widetilde{\mathbb D}^\vae(v)\nabla_v G(t,v)\big)=\mathcal O(\vae^3),
 \end{equation*}
 while what follows from equations (\ref {qasilinear}) is the estimate
   \begin{equation}
   \del_tG(t,v)=\mathcal O(\vae^2) \label{2order}\,.
   \end{equation} 
   Since $\lambda(t) $ and of course $G(t)$ itself
   are analytic functions, from (\ref{2order}) and (\ref{ansatz}), one deduces 
\begin{equation*}
\del_t \tilde h + v\cdot \nabla_x \tilde h + E[\tilde  h]\cdot \nabla_v
G(t) = O(\vae^2)\,.
\end{equation*}
Hence, one has also
\begin{equation*}
 \label{approx}
 \del_t(h- \tilde h) + v\cdot \nabla_x(h- \tilde h )+ E[h-\tilde h ]
\cdot \nabla_v  G   = -\vae \nabla_v \cdot\Big(E[h]   h-\fint E[h]  h
 dx\Big) +\mathcal O(\vae^2) \,.
 \end{equation*}
 Then, with   $ (h- \tilde h)(0,x,v)=0 $, one has also $ h(t)-\tilde h
 (t) =\mathcal O(\vae)\,, $ which eventually implies
\begin{equation*}
  \del_t G +\vae^2 \nabla_v\cdot\Big(\fint E[\tilde h] \tilde h dx\Big)
  =\vae^2 \nabla_v \cdot \fint \Big(E[\tilde h] \tilde h dx  -    E[h] h
  \Big)dx  =\mathcal O(\vae^3)\,.
\end{equation*}
  As pointed above, if $(\lambda, k)$ are solutions of the dispersion 
  equation, then the same is true for $(\lambda^*, -k)\,$ and one can
  extend the above comparison between genuine solutions 
  with  initial data
 \begin{equation}
 \fe(0,x,v) = G(v) +\vae \Re (h(0,x,v))\,, \label{initialdata}
 \end{equation}
 and the approximate solutions given by
 \begin{equation*}
 \tilde \fe(t,x,v) = G(t,v) +\vae \Re \tilde h(t,x,v)\,.
 \end{equation*}
 Since the functions $\tilde h$  satisfy the relation
 \begin{equation}
  (\lambda+ ik\cdot v) \tilde h(t,k,v)+ E[\tilde h(t,k,v)]\cdot \nabla_v G(t,v)=0\,,
 \end{equation}
 one obtains, for solutions with initial data given by
 (\ref{initialdata}),
 \begin{equation*}
 \begin{aligned}
   &\del_tG(t,v) -\vae^2\nabla_v \cdot\left(\frac{ E(0,k)\otimes E(0,k)\Re \lambda e^{2\Re{ \int_0^t ds\lambda(s)} 
}}{(k\cdot v -\Im\lambda)^2+(\Re \lambda)^2}\nabla_v
   G(t,v)\right) \\
   &=\del_tG(t,v) +\vae^2 \nabla_v\cdot\Re\Big(\fint E[\tilde h] \tilde h dx\Big)=
\mathcal O(\vae^3)\,.
\end{aligned}
 \end{equation*}
% :
% \begin{equation}
%\del_t G(t,v) -\vae^2\nabla_v\cdot \big(\frac{ |E(0,k)|^2 e^{2\Re \lambda t
%}}{(k\cdot v -\Im\lambda)^2+(\Re \lambda)^2})\nabla_v
%G(t,v))=O(\vae^3).
% \end{equation} 
 The above construction can be combined  with the Dunford--Kato formula (\ref{Dunford-Kato}).
 Assuming again,  that the roots of the dispersion relation
 $(\lambda, k(\lambda))$ are simple one obtains, for any initial data and any $\delta >0\,,$
 summing with respect to the solutions
 of the dispersion  equation (\ref{penrose2}) with $\Re \lambda >\delta$ the equation
  \begin{equation*}
 \del_t G(t,v) -\sum_{\Re \lambda>\delta }\vae^2\nabla_v\cdot \left(\frac{
   E(0,k(\lambda))\otimes E(0,k(\lambda))\Re \lambda  e^{2\Re {\int_0^t ds\lambda(s) } }} {(k(\lambda)\cdot v
   -\Im\lambda)^2+(\Re \lambda)^2}\nabla_v G(t,v)\right)\\ =\mathcal O(\vae^3) +
 \mathcal O(\vae^2) e^{\delta t}\,,
 \end{equation*} 
 which is the standard quasilinear approximation.

\section{Remarks and Conclusion}
As said in the introduction,  the thread in this contribution is the
comparison  from a genuine mathematical point of view of different
approaches leading to the quasilinear diffusion approximation.
\begin{remark}
The most natural one was with the introduction of the rescaled
equation. However in such configuration it was shown that the
stochastic scenario, which implies that the electric field is
independent of the density, was almost compulsory.  As such, one could
also start directly from a stochastic flow solution of the  ODEs
\begin{equation*}
\begin{aligned}
 & \frac{d}{dt}  X^\vae (t)=  \frac1{\vae^2} V^\vae(t)\,, \\ &
\frac{d}{dt}  V^\vae
  (t) =-\frac 1\vae E^\vae (t,X^\vae(t))\,,
\end{aligned}
\end{equation*}
with a convenient correlation function $A_\uptau$ leading directly at
the macroscopic level, without the kinetic step in between,  to a
diffusion equation
\begin{equation*}
\del_t\overline{\fe} -\nabla_v\cdot (\mathbb {D} (v)\nabla_v
\overline{\fe})=0\,.
\end{equation*}
 Comparison with the diffusion matrix given also in terms of $A_\uptau$
 by (\ref{diffusionmatrix})  leads in such cases to a closed formula
 for the determination of such diffusion. For an interpretation at the
 level of plasma physics for such relation see reference \cite{BB} and the
original ones  \cite{DU66, Wei69}. Complete proofs may be obtained following the
contributions \cite{KP80, DGL87, EP10, EE03b},

 Weak convergence and introduction of randomness imply that
 analysis should be made on the solution rather than on the
 equation. This leads to the introduction of the Duhamel series which
 may be considered as an avatar of other BBGKY hierarchies. However 
 in \cite{PV03, LV04}, the authors have introduced some decorrelation properties
 valid at any time, which close the Duhamel series at second order.
 This is the road which was followed in this contribution.
 \end{remark}
 \begin{remark}
 The short allusion to the Landau damping was motivated on one hand by
 the comparison with issue of strong, versus weak convergence, to $0$ of
 the electric field in the rescaled equation and on the other hand to
 underline the role of estimates on the charge density $\rho[h]$ under the stability
 condition (\ref{superpenrose}), which allows to extend the resolvent
 beyond the imaginary axis.
 \end{remark}
 \begin{remark}
 The  short time validity of the quasilinear approximation is based on
 even much more formal presentations that can be found in basic plasma
 physics textbooks (see for instance \cite{KT73}  pages 514--517).  Here,
 the symbol $\mathcal O(\vae)$ is used everywhere and for a complete proof this
 should be clarified.
 For short time, systematic use of the Nash--Moser
 theorem should balance the loss derivative (of order $1$ with respect
 to $v$) in  Section~\ref{RMQA}. One should keep in mind the
 striking difference between the rescaled diffusion scenario with an independent
 and non-selfconsistent stochastic vector field, which is the typical model for long time dynamics,
 and the short time diffusion scenario, where the selfconsistent electric field is slaved to the solution
 by the Poisson equation. In fact this two diffusion regimes are present in
 the nonlinear relaxation of the weak warm beam-plasma instability problem.
 Selfconsistent numerical simulations of such problem \cite{BEEB11} confirm the existence of this two
 diffusion regimes, plus a third regime between the two. In this third
 regime, dubbed the ``trapping turbulent regime'' in plasma physics literature,
 nonlinear wave-wave coupling plays an important role. Until now and up to
 our knowledge there is no even a roadmap for a full mathematical description
 of such regime.
 \end{remark}

\section {Acknowledgments}
The first author wishes to thank the Observatoire de la C\^ote d'Azur
and the Laboratoire  J.-L. Lagrange for their hospitality and
financial support.

\end{document}